\documentclass[leqno, 12pt]{article}

\usepackage{amsthm}
\usepackage{amssymb}
\usepackage{amsmath}
\usepackage[english]{babel}
\usepackage[utf8]{inputenc}
\usepackage[T1]{fontenc}
\usepackage[normalem]{ulem}
\usepackage{graphicx}
\usepackage{afterpage}
\usepackage{calc}
\usepackage{caption}
\usepackage{subcaption}
\usepackage{times}
\usepackage{fullpage}
\usepackage{float}
\usepackage{marvosym}
\usepackage{hyperref}
\usepackage{tikz}
\usetikzlibrary{shapes,fit,positioning,calc,matrix,arrows}
\tikzset{
  optree/.style={scale=.5,thick,grow'=up,level distance=10mm,inner sep=1pt},
  comp/.style={draw=none,circle,fill,line width=0,inner sep=0pt},
  dot/.style={draw,circle,fill,inner sep=0pt,minimum width=3pt},
  circ/.style={draw,circle,inner sep=1pt,minimum width=4mm},
  emptycirc/.style={draw,circle,inner sep=1pt,minimum width=2mm},  
  root/.style={level distance=10mm,inner sep=1pt},
  leaf/.style={draw=none,circle,fill,line width=0,inner sep=0pt},
  nodot/.style={draw,circle,inner sep=1pt},
}
\usepackage[nottoc, notlof, notlot]{tocbibind}

\newtheorem*{Theorem*}{Theorem}
\newtheorem{Theorem}{Theorem}[section]
\newtheorem{Corollary}[Theorem]{Corollary}
\newtheorem{Proposition}{Proposition}[section]
\newtheorem{Lemma}{Lemma}[section]

\theoremstyle{definition}
\newtheorem{Definition}[Proposition]{Definition}
\newtheorem{Definition/Proposition}[Proposition]{Definition/Proposition}

\theoremstyle{remark}
\newtheorem{Remark}{Remark}[section]
\newtheorem{Example}[Remark]{Example}

\DeclareFontFamily{U}{wncy}{}
\DeclareFontShape{U}{wncy}{m}{n}{<->wncyr10}{}
\DeclareSymbolFont{mcy}{U}{wncy}{m}{n}
\DeclareMathSymbol{\Sh}{\mathord}{mcy}{"58}

\newcommand{\G}{\mathcal{G}}
\newcommand{\F}{\mathcal{F}}

\newcommand{\T}{\mathcal{T}}
\newcommand{\Q}{\mathcal{Q}}
\newcommand{\R}{\mathcal{R}}
\newcommand{\X}{\mathcal{X}}
\newcommand{\Y}{\mathcal{Y}}
\renewcommand{\P}{\mathcal{P}}

\newcommand{\N}{\mathbb{N}}
\newcommand{\Z}{\mathcal{Z}}
\newcommand{\gr}{\mathrm{gr}}

\newcommand{\greg}{\mathrm{Greg}}

\newcommand{\Zin}{\mathrm{Zinb}}
\newcommand{\Ass}{\mathrm{Ass}}
\newcommand{\Pois}{\mathrm{Poiss}}
\newcommand{\PL}{\mathrm{PreLie}}
\newcommand{\Lie}{\mathrm{Lie}}
\newcommand{\CycLie}{\mathrm{CycLie}}
\newcommand{\Com}{\mathrm{Com}}

\newcommand{\Mult}{\textnormal{Mult}}
\newcommand{\id}{\textnormal{id}}
\newcommand{\rev}{\textnormal{rev}}
\newcommand{\tx}{\tilde{x}}

\newcommand{\tg}{\tilde{g}}
\newcommand{\oeis}[1]{\href{https://oeis.org/#1}{{#1}}}

\makeatletter
\newcommand{\oset}[2]{{\mathpalette\o@set{{#1}{#2}}}}
\newcommand{\o@set}[2]{\o@@set{#1}#2}
\newcommand{\o@@set}[3]{%
  \vbox{\offinterlineskip
    \ialign{\hfil##\hfil\cr
      $\m@th\o@set@demote{#1}#2$\cr
      \noalign{\vskip0.2pt}
      $\m@th#1#3$\cr
    }%
  }%
}
\newcommand{\o@set@demote}[1]{%
  \ifx#1\displaystyle\scriptstyle\else
  \ifx#1\textstyle\scriptstyle\else
  \scriptscriptstyle\fi\fi
}
\makeatother

\newlength{\depthofsumsign}
\newcommand{\resum}[1][1.4]{
    \mathop{%
        \raisebox
            {-#1\depthofsumsign+1\depthofsumsign}
            {\scalebox
                {#1}
                {$\displaystyle\sum$}%
            }
    }
}

\newlength\myheight
\newcommand*\ccircled[1]{\settowidth{\myheight}{#1}%
    \raisebox{-.1\myheight}{\tikz[baseline=(char.base)]{%
        \node[shape=circle,draw,minimum size=\myheight*\myheight*.3,inner sep=0.5pt](char){$#1$};}}}

\tikzset{cross/.style={cross out, draw=black, minimum size=2*(#1-\pgflinewidth), inner sep=0pt, outer sep=0pt},
cross/.default={1pt}}

\begin{document}

\title{Combinatorics of pre-Lie products sharing a Lie bracket}
\author{\textsc{Paul Laubie}}
\date{}
\maketitle

\begin{abstract}
We study in detail the operad controlling several pre-Lie algebra structures sharing the same Lie bracket. 
Specifically, we show that this operad admits a combinatorial description similar to that of Chapoton and 
Livernet for the pre-Lie operad, and that it has many of the remarkable algebraic properties of the 
pre-Lie operad.  
\end{abstract}

\section*{Introduction}

Historically, the first example of a non-associative algebraic structure that has been studied is Lie
algebras. They were introduced by Sophus Lie in the 1870s as an algebraic structure on the tangent space
of Lie groups, and were later generalized to capture the algebraic structure of vector fields of any smooth 
manifold.
Although it is well known that the commutator of any associative algebra is a Lie bracket, the above example
of Lie algebra
never comes from a structure of associative algebra on the vector fields. However, a flat 
torsion-free connection on a manifold induces a pre-Lie algebra (or left-symmetric algebra) structure on the
vector fields such that the commutator is the usual Lie bracket on the vector fields \cite{LSA}.
This example may hint that pre-Lie algebras appear more naturally than associative algebras 
when studying Lie algebras. Since flat torsion-free connections
give rise to interesting structures on vector fields, it is not unreasonable to consider the
situation when more than one connection is defined.
This was studied in \cite{D2pl} with the notion of 
Joyce structure. In the case of a manifold with two flat torsion-free connections, the algebraic 
structure on the vector 
fields is richer than a pre-Lie algebra. Indeed, each connection gives a pre-Lie product and
the commutator of those pre-Lie products is the usual Lie bracket on the vector fields.
Hence, we get an algebra with two pre-Lie products sharing the same Lie bracket.
\paragraph{}
If one uses the language of operads to talk about algebraic structures, the operad corresponding to 
algebras with two pre-Lie products sharing the same Lie bracket
is $\bigvee_\Lie^2\PL$, this is the fibered coproduct of two copies of $\PL$ over $\Lie$, i.e.\ the colimit of the
cospan $\PL \leftarrow \Lie \rightarrow \PL$ in the category of algebraic operads. The notion of fiber
coproduct is analogous to the notion of amalgamated sum in group theory, we consider two copies of the $\PL$
and identify the $\Lie$ suboperad of those two copies.
The morphism $\Lie\rightarrow\PL$ used to define this coproduct corresponds to the fact that the operator
$[a, b]=ab-ba$ makes every pre-Lie algebra into a Lie algebra.
Computations show that dimensions of the low arity components of $\bigvee_\Lie^2\PL$ coincides with 
the number of rooted Greg trees \oeis{A005264} in the OEIS \cite{Oeis}. 
A rooted Greg tree is a rooted tree with black and white vertices such that white vertices
are distinguished (e.g.\ numbered), black vertices are undistinguished and each black vertex has at 
least two children. This raises the question whether
the underlying species of $\bigvee_\Lie^2\PL$ is the species of rooted Greg trees. 
We answer this in the affirmative, and in fact prove that this identification
arises from a deformation theoretic phenomenon.
It is known that the operads $\Lie$ and $\PL$ share agreeable algebraic properties, indeed they are 
binary quadratic Koszul operads \cite{AlgOp,PreLie}, 
they have the Nielsen-Schreier property, which means that any subalgebra of a free Lie
(resp. pre-Lie) algebra is free, see \cite{FreeLie,FreeLie2,NSPrt}. Moreover, 
$\PL$ is free over $\Lie$ as a left and right module, see \cite{PLrightFree,PBW} 
(and not as a bimodule). It is natural to ask if $\bigvee_\Lie^2\PL$ shares those properties. We will
show that it is indeed the case. \\
Before we summarize the main results of this paper, remark that the operad $\bigvee_\Lie^2\PL$ can 
naturally be generalized to $\bigvee_\Lie^n\PL$ for any $n\in\N$. 
The operad $\bigvee_\Lie^n\PL$ is
the fibered coproduct of $n$ copies of $\PL$ over $\Lie$. This operad encodes algebras with $n$ 
pre-Lie products that share a common Lie bracket.
\paragraph{Organization of the paper}
We begin by recalling some general facts about operads and more 
particularly about the operads $\Lie$ and 
$\PL$ in Section \ref{seq:recall}. We then give an explicit description of the operadic structure on the
rooted trees from \cite{PreLie} in a way that will make it easier to generalize. In the next section, we
generalize this construction to a larger class of trees, the rooted Greg trees, introducing a naive
operadic structure on them, the operad $\greg$. This newly defined operad is not isomorphic to
$\bigvee_\Lie^2\PL$ but $\bigvee_\Lie^2\PL$ will be shown to be a deformation of $\greg$. We show that:
\begin{Theorem*}\ref{Theorem:m1}
  The operad $\greg$ is binary, quadratic and Koszul.
\end{Theorem*}
In Section \ref{seq:GGO}, we generalize the construction of $\greg$ by making the operadic structure
depend on a coalgebra $C$, defining the operad $\greg^C$. In the next section, we show that we can obtain the 
operad $\bigvee_\Lie^n\PL$ this way, with a careful choice of the coalgebra. We show that:
\begin{Theorem*}\ref{Theorem:m2}
  The operad $\greg^C$ is binary, quadratic and Koszul.
\end{Theorem*}
In Section \ref{seq:free}, we address the question of the freeness of $\greg^C$, and we show that for $C'$ 
a subcoalgebra of $C$:
\begin{Theorem*}\ref{Theorem:m3}{\&}\ref{Theorem:m3ns}
  The operad $\greg^C$ is free as a left $\greg^{C'}$-module, as a right $\greg^{C'}$-module and 
  has the Nielsen-Schreier property.
\end{Theorem*}
In particular:
\begin{Theorem*}\ref{Theorem:m4}{\&}\ref{Theorem:m4ns}
  The operad $\bigvee_\Lie^{n+1}\PL$ is free as a left $\bigvee_\Lie^n\PL$-module, as a right $\bigvee_\Lie^n\PL$-module and 
  has the Nielsen-Schreier property.
\end{Theorem*}
In the next section, we compute explicit generators of $\bigvee_\Lie^{n+1}\PL$ as a left $\bigvee_\Lie^n\PL$-module, generalizing
a conjecture stated in \cite{FPL} proven in \cite{FPRed}. Let $\F$ be the free operad functor, 
$\bar{\F}$ the reduced free operad functor and $\bar{\F}^{(n)}$ the $n$-th iteration of $\bar{\F}$, we show that:
\begin{Theorem*}\ref{Theorem:m5}
  The operad $\bigvee_\Lie^{n+1}\PL$ is isomorphic to $\bigvee_\Lie^n\PL\circ\F(\bar{\F}^{(n)}(\CycLie))$ as a left $\bigvee_\Lie^n\PL$-module.
\end{Theorem*}

\section*{Acknowledgments}

This work was supported by the French national research agency [grant ANR-20-CE40-0016].

I would like to express my deepest gratitude to my PhD advisor Vladimir Dotsenko for his infinite
patience and for encouraging me in this project. 
I am very thankful to Frédéric Chapoton for his availability to discuss
operadic constructions on combinatorial objects. 
I would like to thank Pedro Tamaroff for our discussions about homotopy colimits of operads which lead
me to consider examples of Remark \ref{Remark:PT}.
And I would also 
like to thank Gaétan Leclerc for our numerous and interesting math related discussions.

\tableofcontents

\section{Recollections on the Lie and pre-Lie operads}\label{seq:recall}

The base field is supposed to be of characteristic $0$. We assume that the reader is familiar with the theory 
of species, an introduction and a study of this theory is done in the book of Bergeron, Labelle and Leroux \cite{Species}. We make an extensive use of Gröbner bases over operads, we refer the reader to the book of Loday and Vallette \cite{AlgOp} and the book of Dotsenko and Bremner \cite{AlgComp} for necessary information about operads, shuffle operads and Gröbner bases over operads. We compute explicit Gröbner bases of several operads, because of the size of those bases, lengthy figures are postponed to the \hyperlink{Appendix}{Appendix}, namely figures 2 to 9. The free operad functor is denoted $\F$ and the operad encoding Lie algebras is denoted $\Lie$. It is generated by a binary skew-symmetric operation $l$ satisfying the Jacobi identity. A presentation of this operad is given by:
$$\Lie=\F(l)/\langle l\circ_1l+(l\circ_1l).(1\;2\;3)+(l\circ_1l).(1\;3\;2)\rangle, $$
with $l.(1\;2)=-l$.
\paragraph{}
The operad encoding pre-Lie algebras is denoted $\PL$. It is generated by a binary operation $x$ satisfying the pre-Lie identity. A presentation of this operad is given by:
$$\PL=\F(x)/\langle (x \circ_1x -x \circ_2x) - (x \circ_1x -x \circ_2x).(2\;3) \rangle$$
The main goal of working with shuffle operads is to put a compatible order on the monomials of the free operad to latter define Gröbner bases. It is quite clear that actions of symmetric groups prevent any hope of such order. Hence, the actions of symmetric groups are disposed of when considering shuffle operads, which explain the need of a notation for $x.(1\;2)$, let us write $x.(1\;2)=y$.
\paragraph{} One needs to be careful with the notations since several ``types'' of trees will be used in this article. Shuffle trees are used to write Gröbner bases. Other trees such that rooted trees and rooted Greg trees are used to describe in a combinatorial way the underlying species of the operads.

\begin{Definition}\label{rec:sh}
  Let $\X$ be an alphabet of symbols such that each symbol has an arity. Informally, each symbol represents an operation with a given arity, and a shuffle tree represents compositions of those operations. A \emph{shuffle tree on the alphabet $\X$} is a rooted planar tree such that:
  \begin{itemize}
    \item Internal vertices are labeled by elements of $\X$ and
    have as many children as the arity of their label.
    \item Leaves are labeled by different integers of $\{1, \dots, n\}$, with $n$ the number of leaves.
    \item The numbering of the leaves must satisfy the \emph{local increasing condition}:\\
    The numbering of the leaves is extended to the internal vertices such that each vertex
    receives the smallest number of its children.
    The local increasing condition is that for each internal vertex, the numbering of its children
    is increasing from left to right.
  \end{itemize}
\end{Definition}
In our particular case, the shuffle trees will always be on an alphabet of symbols of arity $2$, meaning the shuffle trees will be binary trees. Let us write the pre-Lie relation with shuffle trees:
\[
\left(
\vcenter{\hbox{\begin{tikzpicture}[scale=0.2]     
    \node (1) at (-3, 6) {};
    \node (2) at (0, 6) {};
    \node (3) at (1.5, 3) {};
    \node (x2) at (-1.5, 3) {};
    \node (x1) at (0, 0) {};
    \draw[thick] (x2)--(1);
    \draw[thick] (x2)--(2);
    \draw[thick] (x1)--(3);
    \draw[thick] (x1)--(x2);
    \node at (1) {\scalebox{0.8}{$1$}};
    \node at (2) {\scalebox{0.8}{$2$}};
    \node at (3) {\scalebox{0.8}{$3$}};
    \node at (x1) {\scalebox{0.8}{$x$}};
    \node at (x2) {\scalebox{0.8}{$x$}};
\end{tikzpicture}}}-
\vcenter{\hbox{\begin{tikzpicture}[scale=0.2]     
    \node (1) at (-1.5, 3) {};
    \node (2) at (0, 6) {};
    \node (3) at (3, 6) {};
    \node (x2) at (1.5, 3) {};
    \node (x1) at (0, 0) {};
    \draw[thick] (x2)--(2);
    \draw[thick] (x2)--(3);
    \draw[thick] (x1)--(1);
    \draw[thick] (x1)--(x2);
    \node at (1) {\scalebox{0.8}{$1$}};
    \node at (2) {\scalebox{0.8}{$2$}};
    \node at (3) {\scalebox{0.8}{$3$}};
    \node at (x1) {\scalebox{0.8}{$x$}};
    \node at (x2) {\scalebox{0.8}{$x$}};
\end{tikzpicture}}} \right) - \left(
\vcenter{\hbox{\begin{tikzpicture}[scale=0.2]     
    \node (1) at (-3, 6) {};
    \node (2) at (0, 6) {};
    \node (3) at (1.5, 3) {};
    \node (x2) at (-1.5, 3) {};
    \node (x1) at (0, 0) {};
    \draw[thick] (x2)--(1);
    \draw[thick] (x2)--(2);
    \draw[thick] (x1)--(3);
    \draw[thick] (x1)--(x2);
    \node at (1) {\scalebox{0.8}{$1$}};
    \node at (2) {\scalebox{0.8}{$3$}};
    \node at (3) {\scalebox{0.8}{$2$}};
    \node at (x1) {\scalebox{0.8}{$x$}};
    \node at (x2) {\scalebox{0.8}{$x$}};
\end{tikzpicture}}}-
\vcenter{\hbox{\begin{tikzpicture}[scale=0.2]     
    \node (1) at (-1.5, 3) {};
    \node (2) at (0, 6) {};
    \node (3) at (3, 6) {};
    \node (x2) at (1.5, 3) {};
    \node (x1) at (0, 0) {};
    \draw[thick] (x2)--(2);
    \draw[thick] (x2)--(3);
    \draw[thick] (x1)--(1);
    \draw[thick] (x1)--(x2);
    \node at (1) {\scalebox{0.8}{$1$}};
    \node at (2) {\scalebox{0.8}{$2$}};
    \node at (3) {\scalebox{0.8}{$3$}};
    \node at (x1) {\scalebox{0.8}{$x$}};
    \node at (x2) {\scalebox{0.8}{$y$}};
\end{tikzpicture}}} \right)
\]
It is a classical result that $\Lie$ is a suboperad of $\PL$ by the inclusion $l\mapsto x-y$, see for example \cite[Proposition 13.4.6]{AlgOp}. The operad encoding two pre-Lie products sharing a Lie bracket is the colimit of the cospan $\PL\leftarrow\Lie\rightarrow\PL$ in the category of operads, it is the coproduct of two copies of $\PL$ fibered over $\Lie$. We note it $\bigvee_\Lie^2\PL$. A presentation of this operad is given by:
\begin{multline*}
\bigvee_\Lie^2\PL=\F(x_1, x_2)/\langle(x_1 \circ_1x_1 -x_1 \circ_2x_1) - (x_1 \circ_1x_1 -x_1 \circ_2x_1).(2\;3)\;;\\
\qquad\qquad(x_2 \circ_1x_2 -x_2 \circ_2x_2) - (x_2 \circ_1x_2 -x_2 \circ_2x_2).(2\;3)\;; 
(x_1-x_2)-(x_1-x_2).(1\;2)\rangle
\end{multline*}
with the notation $x_1.(1\;2)=y_1$ and $x_2.(1\;2)=y_2$. This presentation is not quadratic, indeed the last relation is linear, a quadratic presentation will be given in Section \ref{seq:GGO}. However, this presentation is enough to compute the first arity-wise dimensions of this operad either by hand or with a computer using the Haskell calculator written by Dotsenko and Heijltjes \cite{Hask}.\\
\begin{table*}[h]
\begin{tabular}{|r|c|c|c|c|c|c|c|}
\hline
Arity & $1$ & $2$ & $3$ & $4$ & $5$ & $\cdots$ & $n$ \\
\hline 
\hline
$\dim(\Lie(n))$ & $1$ & $1$ & $2$ & $6$ & $24$ & $\cdots$ & $(n-1)!$\\
\hline
$\dim(\PL(n))$ & $1$ & $2$ & $9$ & $64$ & $625$ & $\cdots$ & $n^{n-1}$ \\
\hline
$\dim\big(\bigvee_\Lie^2\PL(n)\big)$ & $1$ & $3$ & $22$ & $262$ & ? & ? & ? \\
\hline
\end{tabular}
\end{table*}

One may recognize in the last line the first terms of the sequence of the number of rooted Greg trees with $n$ white vertices \oeis{A005264} in \cite{Oeis}. A rooted Greg tree is a rooted tree with black and white vertices such that the black vertices have at least two children. Let us state the following definition:
\begin{Definition}
  Two species $\mathcal{S}$ and $\T$ are \emph{equinumerous} if they have the same sequence of dimensions.
  It is a weaker notion than the isomorphism of species.
\end{Definition}
This leads to the natural following questions:
\begin{itemize}
  \item Are $\bigvee_\Lie^2\PL$ and the species of rooted Greg trees \emph{equinumerous}?
  \item Are they \emph{isomorphic as species}?
  \item Are they \emph{isomorphic with their extra algebraic structure}?
\end{itemize}  
In this article we answer positively to first two questions. We then define a naive operadic structure on the species of rooted Greg trees and show that $\bigvee_\Lie^2\PL$ is a deformation of this newly defined operad. We then show that the operad $\bigvee_\Lie^2\PL$ has agreeable algebraic properties namely is binary, quadratic and Koszul. The operad $\bigvee_\Lie^2\PL$ also has the three following freeness properties: 
\begin{itemize}
  \item It is free as a left $\PL$-module.
  \item It is free as a right $\PL$-module.
  \item It has the Nielsen-Schreier property.
\end{itemize}
Those freeness properties will be discussed in Section \ref{seq:free}.

\section{Operadic structure on rooted trees}\label{seq:RT}

Let us give a full construction of the operadic structure on rooted trees defined by Chapoton and Livernet in \cite{PreLie}, in order to generalize it on rooted Greg trees.

\begin{Definition}
  A rooted tree is a connected non-empty non-oriented graph without cycle with a distinguished vertex 
  called the root. Let $I\subset\N$ be finite, vertices are in bijection with $I$.
  The action of the symmetric group is given by the permutation of the labels.
  When drawing rooted trees, root will be at the bottom and leaves at the top.
\end{Definition}

For $S$ a rooted tree, $v$ a vertex of $S$, the forest $FS=\{S_1, \dots, S_k\}$ the forest of the children of $v$ and $B$ the rooted tree below $v$. Let us introduce the following notation:
\[
  \vcenter{\hbox{\begin{tikzpicture}[scale=0.5]     
    \node (S) at (0, 0) {};
    
    \node[isosceles triangle, 
    isosceles triangle apex angle=60, 
    draw, 
    rotate=270, 
    fill=white!50, 
    minimum size =24pt] (ST) at (S){};



    \draw[circle, fill=black] (ST.east) circle [radius=2pt];

    \node at (S) {\scalebox{0.8}{$S$}};
  \end{tikzpicture}}}
  \qquad
  =
  \qquad
  \vcenter{\hbox{\begin{tikzpicture}[scale=0.5]     
    \node (lv) at (0, 0) {};
    \node (SD) at (0, -3) {};
    \node (SU1) at (-2, 3) {};
    \node (dot) at (0, 3) {};
    \node (SUk) at (2, 3) {};

    \node[isosceles triangle, 
    isosceles triangle apex angle=60, 
    draw, 
    rotate=-90, 
    fill=white!50, 
    minimum size =24pt] (SDT) at (SD){};
    \node[isosceles triangle, 
    isosceles triangle apex angle=60, 
    draw, 
    rotate=270, 
    fill=white!50, 
    minimum size =24pt] (SU1T) at (SU1){};
    \node[isosceles triangle, 
    isosceles triangle apex angle=60, 
    draw, 
    rotate=270, 
    fill=white!50, 
    minimum size =24pt] (SUkT) at (SUk){};

    \draw[thick] (SU1T.east)--(lv);
    \draw[thick] (SUkT.east)--(lv);
    \draw[thick] (lv)--(SDT.west);

    \node at (dot) {$\cdots$};

    \draw[circle, fill=white] (lv) circle [radius=24pt];
    \draw[circle, fill=black] (SDT.east) circle [radius=2pt];
    \draw[circle, fill=black] (SU1T.east) circle [radius=2pt];
    \draw[circle, fill=black] (SUkT.east) circle [radius=2pt];
    
    \node at (lv) {\scalebox{1}{$v$}};

    \node at (SU1) {\scalebox{0.8}{$S_1$}};
    \node at (SUk) {\scalebox{0.8}{$S_k$}};
    \node at (SD) {\scalebox{0.8}{$B$}};
  \end{tikzpicture}}}
  \qquad
  =
  \qquad
  \vcenter{\hbox{\begin{tikzpicture}[scale=0.5]     
    \node (lv) at (0, 0) {};
    \node (SD) at (0, -2.7) {};
    \node (F) at (0, 3) {};
    
    \node[isosceles triangle, 
    isosceles triangle apex angle=60, 
    draw, 
    rotate=270, 
    fill=white!50, 
    minimum size =24pt] (SDT) at (SD){};
    \node[isosceles triangle, 
    isosceles triangle apex angle=60, 
    draw, 
    rotate=270, 
    fill=white!50, 
    minimum size =24pt] (FT) at (F){};

    \draw[thick, double] (FT.east)--(lv);
    \draw[thick] (lv)--(SDT.west);


    \draw[circle, fill=white] (lv) circle [radius=24pt];
    \draw[circle, fill=black] (FT.east) circle [radius=2pt];
    \draw[circle, fill=black] (SDT.east) circle [radius=2pt];
    
    \node at (lv) {\scalebox{1}{$v$}};

    \node at (F) {\scalebox{0.8}{$FS$}};
    \node at (SD) {\scalebox{0.8}{$B$}};
  \end{tikzpicture}}}
\]
We use circles to represent vertices, triangles to represent trees or forests and double edges to represent that each tree of the forest $FS$ is grafted to $v$.

\begin{Definition}\label{Definition:fall}
  Let $S$ and $T$ be two rooted trees labeled over disjoint sets and $V(S)$ the set of vertices of $S$, let $S\star T$ be the \emph{fall product} of $T$ over $S$ defined by:
  \[
  S\star T = \resum[1]_{v\in V(S)}
  \vcenter{\hbox{\begin{tikzpicture}[scale=0.5]     
    \node (lv) at (0, 0) {};
    \node (SD) at (0, -3) {};
    \node (T) at (1.5, 3) {};
    \node (FS) at (-1.5, 3) {};

    \node[isosceles triangle, 
    isosceles triangle apex angle=60, 
    draw, 
    rotate=270, 
    fill=white!50, 
    minimum size =24pt] (SDT) at (SD){};
    \node[isosceles triangle, 
    isosceles triangle apex angle=60, 
    draw, 
    rotate=270, 
    fill=white!50, 
    minimum size =24pt] (TT) at (T){};
    \node[isosceles triangle, 
    isosceles triangle apex angle=60, 
    draw, 
    rotate=270, 
    fill=white!50, 
    minimum size =24pt] (FST) at (FS){};    
    \draw[thick] (TT.east)--(lv);
    \draw[thick, double] (FST.east)--(lv);
    \draw[thick] (lv)--(SDT.west);


    \draw[circle, fill=white] (lv) circle [radius=24pt];
    \draw[circle, fill=black] (TT.east) circle [radius=2pt];
    \draw[circle, fill=black] (FST.east) circle [radius=2pt];
    \draw[circle, fill=black] (SDT.east) circle [radius=2pt];

    \node at (lv) {\scalebox{1}{$v$}};

    \node at (T) {\scalebox{0.8}{$T$}};
    \node at (FS) {\scalebox{0.8}{$FS$}};
    \node at (SD) {\scalebox{0.8}{$B$}};
  \end{tikzpicture}}}
  \]
  For the sake of readability, let us omit the sums, the tree $B$ and the forest $FS$:
  \[
    S\star T  =
    \vcenter{\hbox{\begin{tikzpicture}[scale=0.5]     
      \node (lv) at (0, 0) {};
      \node (T) at (0, 3) {};

      \node[isosceles triangle, 
      isosceles triangle apex angle=60, 
      draw, 
      rotate=270, 
      fill=white!50, 
      minimum size =24pt] (TT) at (T){};
      
      \draw[thick] (TT.east)--(lv);


      \draw[circle, fill=white] (lv) circle [radius=24pt];
      \draw[circle, fill=black] (TT.east) circle [radius=2pt];
      
      \node at (lv) {\scalebox{1}{$v$}};

      \node at (T) {\scalebox{0.8}{$T$}};
    \end{tikzpicture}}}  
  \]
\end{Definition}

\begin{Example}
  Let compute $(R\star S)\star T$ to grasp the definition of the fall product. Let $v_r$ and $v_{r'}$ be generic vertices of $R$ and $v_s$ a generic vertex of $S$.
    \[
      (R\star S)\star T=
      \vcenter{\hbox{\begin{tikzpicture}[scale=0.5]     
        \node (lv) at (0, 0) {};
        \node (r) at (0, -3) {};
        \node (T) at (0, 3) {};

        \node[isosceles triangle, 
        isosceles triangle apex angle=60, 
        draw, 
        rotate=270, 
        fill=white!50, 
        minimum size =24pt] (TT) at (T){};
        
        \draw[thick] (TT.east)--(lv);
        \draw[thick, dashed] (r)--(lv);


        \draw[circle, fill=white] (lv) circle [radius=24pt];
        \draw[circle, fill=white] (r) circle [radius=24pt];
        \draw[circle, fill=black] (TT.east) circle [radius=2pt];
        
        \node at (lv) {\scalebox{1}{$v_s$}};
        \node at (r) {\scalebox{1}{$v_r$}};

        \node at (T) {\scalebox{0.8}{$T$}};
      \end{tikzpicture}}}
      +
      \vcenter{\hbox{\begin{tikzpicture}[scale=0.5]     
        \node (r) at (3, 0) {};
        \node (s) at (0, 0) {};
        \node (S) at (3, 3) {};
        \node (T) at (0, 3) {};

        \node[isosceles triangle, 
        isosceles triangle apex angle=60, 
        draw, 
        rotate=270, 
        fill=white!50, 
        minimum size =24pt] (TT) at (T){};
        \node[isosceles triangle, 
        isosceles triangle apex angle=60, 
        draw, 
        rotate=270, 
        fill=white!50, 
        minimum size =24pt] (ST) at (S){};

        \draw[thick] (TT.east)--(s);
        \draw[thick] (ST.east)--(r);
        \draw[thick, dashed] (r)--(s);


        \draw[circle, fill=white] (s) circle [radius=24pt];
        \draw[circle, fill=white] (r) circle [radius=24pt];
        \draw[circle, fill=black] (TT.east) circle [radius=2pt];
        \draw[circle, fill=black] (ST.east) circle [radius=2pt];
        
        \node at (s) {\scalebox{1}{$v_{r'}$}};
        \node at (r) {\scalebox{1}{$v_r$}};

        \node at (S) {\scalebox{0.8}{$S$}};
        \node at (T) {\scalebox{0.8}{$T$}};
      \end{tikzpicture}}}
      +
      \vcenter{\hbox{\begin{tikzpicture}[scale=0.5]     
        \node (T) at (1.5, 3) {};
        \node (r) at (0, 0) {};
        \node (S) at (-1.5, 3) {};

        \node[isosceles triangle, 
        isosceles triangle apex angle=60, 
        draw, 
        rotate=270, 
        fill=white!50, 
        minimum size =24pt] (TT) at (T){};
        \node[isosceles triangle, 
        isosceles triangle apex angle=60, 
        draw, 
        rotate=270, 
        fill=white!50, 
        minimum size =24pt] (ST) at (S){};
        
        \draw[thick] (TT.east)--(r);
        \draw[thick] (ST.east)--(r);


        \draw[circle, fill=white] (r) circle [radius=24pt];
        \draw[circle, fill=black] (TT.east) circle [radius=2pt];
        \draw[circle, fill=black] (ST.east) circle [radius=2pt];
        
        \node at (r) {\scalebox{1}{$v_r$}};

        \node at (S) {\scalebox{0.8}{$S$}};
        \node at (T) {\scalebox{0.8}{$T$}};
      \end{tikzpicture}}}
    \]
  Here $S$ fall of the vertex $v_r$ of $R$, then either $T$ fall on a vertex of $S$, or $T$ fall on another vertex $v_{r'}$ of $R$, or $T$ fall on the vertex $v_r$ of $R$.\\
  The dotted edges represent a path between the two vertices that may be longer than one edge. The dotted edge is horizontal if the two vertices can be one below the other, one above the other or neither of them.
\end{Example}

\begin{Proposition}
  The fall product is pre-Lie.
\end{Proposition}

\begin{proof}
  We have computed $(R\star S)\star T$, moreover we have:
  \[
    R\star (S\star T)=
    \vcenter{\hbox{\begin{tikzpicture}[scale=0.5]     
      \node (lv) at (0, 0) {};
      \node (r) at (0, -3) {};
      \node (T) at (0, 3) {};

      \node[isosceles triangle, 
      isosceles triangle apex angle=60, 
      draw, 
      rotate=270, 
      fill=white!50, 
      minimum size =24pt] (TT) at (T){};
      
      \draw[thick] (TT.east)--(lv);
      \draw[thick, dashed] (r)--(lv);


      \draw[circle, fill=white] (lv) circle [radius=24pt];
      \draw[circle, fill=white] (r) circle [radius=24pt];
      \draw[circle, fill=black] (TT.east) circle [radius=2pt];
      
      \node at (lv) {\scalebox{1}{$v_s$}};
      \node at (r) {\scalebox{1}{$v_r$}};

      \node at (T) {\scalebox{0.8}{$T$}};
    \end{tikzpicture}}}
  \]
  Hence, $(R\star S)\star T-R\star (S\star T)$ is symmetric in $S$ and $T$. Hence, the fall product is pre-Lie.
\end{proof}

\begin{Remark}
  The fall product allows us to graft a rooted tree $T$ over another rooted tree $S$ on all 
  possible vertices of $S$. However, as we can see in the above computation, a naive composition
  of the fall product is not enough to make several trees fall on the same tree.
  Indeed, when computing $(R\star S)\star T$, we have that $S$ fall on $R$, but $T$ can
  either fall on $S$ or on $R$. 
  The solution is to use the symmetric brace products.
\end{Remark}

The symmetric brace products were first introduced by Lada and Markl in \cite{sBr} and the following formula to get the symmetric brace products from a pre-Lie product was given by Oudom and Guin in \cite{BRPL}:
\begin{itemize}
  \item $Br(S)=S$
  \item $Br(S;T)=S\star T$
  \item $Br(S;T_1, \dots, T_{n+1})=Br(S;T_1, \dots, T_n)\star T_{n+1}-
  \sum_{i=1}^nBr(S;T_1, \dots, T_i\star T_{n+1}, \dots, T_n)$
\end{itemize}
The symmetric brace product $Br(S;T_1, \dots, T_n)$ is the sum of all possible ways to graft the trees $T_1, \dots, T_n$ on vertices of $S$. It is symmetric in the $T_i$'s. For $FT=\{T_1, \dots, T_n\}$, we will write $Br(S;FT)$.
\paragraph{}
Let us recall from \cite[Definition 5.3.7]{AlgOp} that an operad structure on a species can be given by a collection of operations $\circ_i$ of arity $2$, the partial compositions satisfying the sequential and parallel composition axioms. Let us define these operations on rooted trees using the symmetric brace products.

\begin{Definition}\label{Definition:inf}
  Let $T$ and $S$ be two rooted trees, let $i$ be a label of a vertex of $T$ and $v$ this vertex.
  \[
    \vcenter{\hbox{\begin{tikzpicture}[scale=0.5]     
      \node (T) at (0, 0) {};

      \node[isosceles triangle, 
      isosceles triangle apex angle=60, 
      draw, 
      rotate=270, 
      fill=white!50, 
      minimum size =24pt] (TT) at (T){};      



      \draw[circle, fill=black] (TT.east) circle [radius=2pt];
      
      \node at (T) {\scalebox{0.8}{$T$}};

    \end{tikzpicture}}}
    =
    \vcenter{\hbox{\begin{tikzpicture}[scale=0.5]     
      \node (FT) at (0, 3) {};
      \node (i) at (0, 0) {};
      \node (TD) at (0, -3) {};

      \node[isosceles triangle, 
      isosceles triangle apex angle=60, 
      draw, 
      rotate=270, 
      fill=white!50, 
      minimum size =24pt] (FTT) at (FT){};
      \node[isosceles triangle, 
      isosceles triangle apex angle=60, 
      draw, 
      rotate=270, 
      fill=white!50, 
      minimum size =24pt] (TDT) at (TD){};

      \draw[thick, double] (FTT.east)--(i);
      \draw[thick] (TDT.west)--(i);


      \draw[circle, fill=white] (i) circle [radius=24pt];
      \draw[circle, fill=black] (FTT.east) circle [radius=2pt];
      \draw[circle, fill=black] (TDT.east) circle [radius=2pt];
      
      \node at (i) {\scalebox{1}{$v$}};
      \node at (FT) {\scalebox{0.8}{$FT$}};
      \node at (TD) {\scalebox{0.8}{$B$}};

    \end{tikzpicture}}}
  \]
  Let us consider $U=Br(S;FT)$, then the partial composition $\circ_i$ is defined by:
  \[
    T\circ_i S = 
    \vcenter{\hbox{\begin{tikzpicture}[scale=0.5]     
      \node (T) at (0, 0) {};
      \node (i) at (0, 3) {};
      \node (S) at (0, 6) {};

      \node[isosceles triangle, 
      isosceles triangle apex angle=60, 
      draw, 
      rotate=270, 
      fill=white!50, 
      minimum size =24pt] (TT) at (T){};
      \node[isosceles triangle, 
      isosceles triangle apex angle=60, 
      draw, 
      rotate=270, 
      fill=white!50, 
      minimum size =24pt] (ST) at (S){}; 

      \draw[thick] (TT.west)--(ST.east);


      \draw[circle, fill=black] (TT.east) circle [radius=2pt];
      \draw[circle, fill=black] (ST.east) circle [radius=2pt];
      \draw[circle, fill=white] (i) circle [radius=24pt];

      \node[cross=13pt] at (i) {};
      
      \node at (T) {\scalebox{0.8}{$B$}};
      \node at (i) {\scalebox{0.8}{$v$}};
      \node at (S) {\scalebox{0.8}{$U$}};

    \end{tikzpicture}}}
  \]
  The children of $v$ fall on $S$, then the result of this symmetric brace product is grafted on $v$, and finally the vertex $v$ is removed. This is exactly the insertion of $S$ in $T$ at the vertex $v$, with all the children of $v$ falling on $S$. Although this construction is not described the same way as in \cite{PreLie}, the two constructions are the same. Hence, those partial compositions satisfy the sequential and parallel composition axioms.
\end{Definition}

\begin{Remark}
  When defining partial compositions on a species $P$, we define
  $$\circ_i:\P(A\sqcup\{i\})\otimes\P(B)\rightarrow\P(A\sqcup B), $$ 
  where $A$ and $B$ are disjoint sets.
  It is equivalent to the definition with 
  $$\circ_i:\P(n)\otimes\P(m)\rightarrow\P(n+m-1)$$
  which involves some renumbering.
\end{Remark}

Let us recall the following result:

\begin{Theorem}\cite[Theorem 1.9]{PreLie}
  Let $\mathcal{RT}$ be the species of rooted trees. The operad $(\mathcal{RT}, \{\circ_i\})$ with the   partial compositions defined as above is isomorphic to the operad $\PL$. Moreover, the isomorphism is   $\vcenter{\hbox{\begin{tikzpicture}[scale=0.2]     
      \node (1) at (0, 0) {};
      \node (2) at (0, 3) {};


      \draw[thick] (1)--(2);


      \draw[circle, fill=white] (1) circle [radius=24pt];
      \draw[circle, fill=white] (2) circle [radius=24pt];
      
      \node at (1) {\scalebox{0.8}{$1$}};
      \node at (2) {\scalebox{0.8}{$2$}};

  \end{tikzpicture}}}\mapsto x$ with $x$ the generator of $\PL$.
\end{Theorem}

\section{Rooted Greg trees and the Greg operad}\label{seq:RGTO}

Greg trees were introduced by Greg \cite{Hgreg} and Maas \cite{Hmaas} in order to state a combinatorial problem arising from textual criticism, which was solved by Flight in \cite{Greg}.\\

\begin{Definition}
  A \emph{rooted Greg tree} $T$ is a rooted tree with two colors of vertices black and white such that:
  \begin{itemize}
    \item the white vertices are in bijection with $I$ a finite set of labels, the number of white vertices is
    the \emph{arity} of $T$;
    \item the black vertices are unlabeled and have at least two children, the number of black vertices is
    the \emph{weight} of $T$.
  \end{itemize}  
\end{Definition} 

Let us introduce the following notation:
\begin{itemize}
  \item $WV(T)$ is the set of white vertices of $T$, $BV(T)$ the set of black vertices of $T$ and $V(T)$
  the set of vertices of $T$.
  \item $G_k(n)$ is the set of Greg trees of arity $n$ with white vertices labeled by 
  $\{1, \dots, n\}$ and weight $k$.
  \item $G(n)$ is the set of Greg trees of arity $n$ with white vertices labeled by 
  $\{1, \dots, n\}$.
  \item $\G_k(n)$ is the vector space spanned by $G_k(n)$, $\G(n)$ is
  the vector space spanned by $G(n)$, and $\G$ is the species of Greg trees.
\end{itemize}
The action of the symmetric group on rooted Greg trees is the natural one permuting the labels of the white vertices.\\

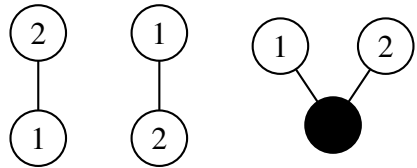
\begin{figure}[ht]
  \caption{Greg trees of arity $n=2$}
  \label{fig:gregtrees2}
   \[
    \vcenter{\hbox{\begin{tikzpicture}[scale=0.7]
      \draw[thick] (0, 0)--(0, 2);
   
      \draw[fill=white, thick] (0, 0) circle [radius=15pt];
      \draw[fill=white, thick] (0, 2) circle [radius=15pt];
      \node at (0, 0) {\scalebox{1}{1}};
      \node at (0, 2) {\scalebox{1}{2}};
    \end{tikzpicture}}}
    \qquad
      \vcenter{\hbox{\begin{tikzpicture}[scale=0.7]
        \draw[thick] (0, 0)--(0, 2);
     
        \draw[fill=white, thick] (0, 0) circle [radius=15pt];
        \draw[fill=white, thick] (0, 2) circle [radius=15pt];
        \node at (0, 0) {\scalebox{1}{2}};
        \node at (0, 2) {\scalebox{1}{1}};
      \end{tikzpicture}}}
      \qquad
      \vcenter{\hbox{\begin{tikzpicture}[scale=0.7]
        \draw[thick] (0, 0.5)--(-1, 2);
        \draw[thick] (0, 0.5)--(1, 2);
       
        \draw[fill=black, thick] (0, 0.5) circle [radius=15pt];
        \draw[fill=white, thick] (-1, 2) circle [radius=15pt];
        \draw[fill=white, thick] (1, 2) circle [radius=15pt];
        \node at (0, 0.5) { };
        \node at (-1, 2) {\scalebox{1}{1}};
        \node at (1, 2) {\scalebox{1}{2}};
      \end{tikzpicture}}}
   \]
\end{figure}

\begin{Remark}
  The condition that the black vertices have at least two children ensures that $G(n)$ is finite. Another important fact is that $(G_0(n))_{n\in\N}$ is the species of rooted trees. Hence, in order to define an operad structure on the rooted Greg trees, one may try to mimic and generalize the construction on rooted trees of \cite{PreLie}.
\end{Remark}

\begin{Proposition}
  Let $f_\G(t)$ be the exponential generating series of $\G$. Then $f_\G(t)$ is the inverse under composition of $(2t+1)\exp(-t)-1$.
\end{Proposition}

Let us naively generalize the construction described in the last section to the rooted Greg trees. Definition \ref{Definition:fall} of the fall product is the same, it is straightforward to check that it is a pre-Lie product on $\G$. Definition \ref{Definition:inf} of partial compositions is also the same. However, one may want to check that it satisfies the sequential and parallel composition axioms stated in \cite[Definition 5.3.7]{AlgOp}.

\begin{Proposition}\label{Proposition:inf}
  The partial compositions on $\G$ satisfy the sequential and parallel composition axioms.
\end{Proposition}

\begin{proof}
  \emph{Parallel composition axiom:}
  Let us compute $(T\circ_i S)\circ_j R$ in the case where $i$ and $j$ are the labels $v_i$ and $v_j$ which are white vertex of $T$. Let us write $T$ the following way:
  \[
  T=
  \vcenter{\hbox{\begin{tikzpicture}[scale=0.5]     
      \node (T) at (0, 0) {};
      \node (i) at (-2, 2.5) {};
      \node (j) at (2, 2.5) {};
      \node (S) at (2, 6) {};
      \node (R) at (-2, 6) {};

      \node[isosceles triangle, 
      isosceles triangle apex angle=60, 
      draw, 
      rotate=270, 
      fill=white!50, 
      minimum size =24pt] (TT) at (T){};
      \node[isosceles triangle, 
      isosceles triangle apex angle=60, 
      draw, 
      rotate=270, 
      fill=white!50, 
      minimum size =32pt] (ST) at (S){};
      \node[isosceles triangle, 
      isosceles triangle apex angle=60, 
      draw, 
      rotate=270, 
      fill=white!50, 
      minimum size =32pt] (RT) at (R){};
      
      \draw[thick] (TT.west)--(i);
      \draw[thick] (TT.west)--(j);
      \draw[thick, double] (RT.east)--(i);
      \draw[thick, double] (ST.east)--(j);


      \draw[circle, fill=white] (i) circle [radius=24pt];
      \draw[circle, fill=white] (j) circle [radius=24pt];
      \draw[circle, fill=black] (TT.east) circle [radius=2pt];
      \draw[circle, fill=black] (ST.east) circle [radius=2pt];
      \draw[circle, fill=black] (RT.east) circle [radius=2pt];
      

      \node at (T) {\scalebox{0.8}{$T_0$}};
      \node at (R) {\scalebox{0.8}{$FT^{(i)}$}};
      \node at (i) {\scalebox{0.8}{$v_i$}};
      \node at (j) {\scalebox{0.8}{$v_j$}};
      \node at (S) {\scalebox{0.8}{$FT^{(j)}$}};
    \end{tikzpicture}}}
  \]
  Applying the definition of the partial composition, we get:
  \[
  (T\circ_i S)\circ_j R =
  T=
  \vcenter{\hbox{\begin{tikzpicture}[scale=0.5]     
      \node (T) at (0, 0) {};
      \node (i) at (-2, 2.5) {};
      \node (j) at (2, 2.5) {};
      \node (S) at (2, 6) {};
      \node (R) at (-2, 6) {};

      \node[isosceles triangle, 
      isosceles triangle apex angle=60, 
      draw, 
      rotate=270, 
      fill=white!50, 
      minimum size =24pt] (TT) at (T){};
      \node[isosceles triangle, 
      isosceles triangle apex angle=60, 
      draw, 
      rotate=270, 
      fill=white!50, 
      minimum size =24pt] (ST) at (S){};
      \node[isosceles triangle, 
      isosceles triangle apex angle=60, 
      draw, 
      rotate=270, 
      fill=white!50, 
      minimum size =24pt] (RT) at (R){};
      
      \draw[thick] (TT.west)--(i);
      \draw[thick] (TT.west)--(j);
      \draw[thick] (RT.east)--(i);
      \draw[thick] (ST.east)--(j);


      \draw[circle, fill=white] (i) circle [radius=24pt];
      \draw[circle, fill=white] (j) circle [radius=24pt];
      \draw[circle, fill=black] (TT.east) circle [radius=2pt];
      \draw[circle, fill=black] (ST.east) circle [radius=2pt];
      \draw[circle, fill=black] (RT.east) circle [radius=2pt];
      
      \node[cross=13pt] at (i) {};
      \node[cross=13pt] at (j) {};


      \node at (T) {\scalebox{0.8}{$T_0$}};
      \node at (R) {\scalebox{0.8}{$U$}};
      \node at (i) {\scalebox{0.8}{$v_i$}};
      \node at (j) {\scalebox{0.8}{$v_j$}};
      \node at (S) {\scalebox{0.8}{$V$}};
    \end{tikzpicture}}}
    = (T\circ_j R)\circ_i S
  \]
  with $U=Br(S;FT^{(i)})$ and $V=Br(R;FT^{(j)})$.
  The parallel axiom is verified.

  \emph{Sequential composition axiom:}
  Let us compute $T\circ_i (S\circ_j R)$ in the case where $i$ is the label of $v_i$ a white vertex of $T$ and $j$ is the label of $v_j$ a white vertex of $S$. Let us write:
  \[
  T=
  \vcenter{\hbox{\begin{tikzpicture}[scale=0.5]     
      \node (T) at (0, 0) {};
      \node (i) at (0, 3) {};
      \node (R) at (0, 6) {};

      \node[isosceles triangle, 
      isosceles triangle apex angle=60, 
      draw, 
      rotate=270, 
      fill=white!50, 
      minimum size =24pt] (TT) at (T){};
      \node[isosceles triangle, 
      isosceles triangle apex angle=60, 
      draw, 
      rotate=270, 
      fill=white!50, 
      minimum size =24pt] (RT) at (R){};

      \draw[thick] (TT.west)--(i);
      \draw[thick, double] (RT.east)--(i);


      \draw[circle, fill=white] (i) circle [radius=24pt];
      \draw[circle, fill=black] (RT.east) circle [radius=2pt];
      \draw[circle, fill=black] (TT.east) circle [radius=2pt];
      

      \node at (T) {\scalebox{0.8}{$T_0$}};
      \node at (R) {\scalebox{0.8}{$FT$}};
      \node at (i) {\scalebox{0.8}{$v_i$}};
    \end{tikzpicture}}}
    \qquad
    \text{and}
    \qquad
    S=
    \vcenter{\hbox{\begin{tikzpicture}[scale=0.5]     
      \node (T) at (0, 0) {};
      \node (i) at (0, 3) {};
      \node (R) at (0, 6) {};

      \node[isosceles triangle, 
      isosceles triangle apex angle=60, 
      draw, 
      rotate=270, 
      fill=white!50, 
      minimum size =24pt] (TT) at (T){};
      \node[isosceles triangle, 
      isosceles triangle apex angle=60, 
      draw, 
      rotate=270, 
      fill=white!50, 
      minimum size =24pt] (RT) at (R){};

      \draw[thick] (TT.west)--(i);
      \draw[thick, double] (RT.east)--(i);


      \draw[circle, fill=white] (i) circle [radius=24pt];
      \draw[circle, fill=black] (RT.east) circle [radius=2pt];
      \draw[circle, fill=black] (TT.east) circle [radius=2pt];
      

      \node at (T) {\scalebox{0.8}{$S_0$}};
      \node at (R) {\scalebox{0.8}{$FS$}};
      \node at (i) {\scalebox{0.8}{$v_j$}};
    \end{tikzpicture}}}
  \]
  Then we have:
  \[
  T\star S=
  \vcenter{\hbox{\begin{tikzpicture}[scale=0.5]     
    \node (T) at (0, -3) {};
    \node (c) at (0, 0) {};
    \node (i) at (0, 6) {};
    \node (R) at (0, 3) {};
    \node (FU) at (-2, 9) {};
    \node (FT) at (2, 9) {};

    \node[isosceles triangle, 
    isosceles triangle apex angle=60, 
    draw, 
    rotate=270, 
    fill=white!50, 
    minimum size =24pt] (TT) at (T){};
    \node[isosceles triangle, 
    isosceles triangle apex angle=60, 
    draw, 
    rotate=270, 
    fill=white!50, 
    minimum size =24pt] (RT) at (R){};
    \node[isosceles triangle, 
    isosceles triangle apex angle=60, 
    draw, 
    rotate=270, 
    fill=white!50, 
    minimum size =32pt] (FUT) at (FU){};
    \node[isosceles triangle, 
    isosceles triangle apex angle=60, 
    draw, 
    rotate=270, 
    fill=white!50, 
    minimum size =32pt] (FTT) at (FT){};

    \draw[thick] (TT.west)--(c);
    \draw[thick] (RT.east)--(c);
    \draw[thick] (RT.west)--(i);
    \draw[thick, double] (FUT.east)--(i);
    \draw[thick, double] (FTT.east)--(i);


    \draw[circle, fill=white] (i) circle [radius=24pt];
    \draw[circle, fill=white] (c) circle [radius=24pt];
    \draw[circle, fill=black] (RT.east) circle [radius=2pt];
    \draw[circle, fill=black] (TT.east) circle [radius=2pt];
    \draw[circle, fill=black] (FUT.east) circle [radius=2pt];
    \draw[circle, fill=black] (FTT.east) circle [radius=2pt];

    \node[cross=13] at (c) {};
    

    \node at (T) {\scalebox{0.8}{$T_0$}};
    \node at (R) {\scalebox{0.8}{$U_0$}};
    \node at (FU) {\scalebox{0.8}{$FU$}};
    \node at (FT) {\scalebox{0.8}{$FT_{k+1}$}};
    \node at (i) {\scalebox{0.8}{$v_j$}};
    \node at (c) {\scalebox{0.8}{$v_i$}};
  \end{tikzpicture}}} 
  \]
  with $FS=(S_1, \dots, S_k)$, $FU=(U_1, \dots, U_k)$, $U_l=Br(S_l;FT_l)$ and $FT=\bigsqcup_{l=0}^{k+1} FT_l$ for every possible such decomposition of $FT$ in (possibly empty) sub-forest. Hence:
  \[
    (T\star S)\star R=
    \vcenter{\hbox{\begin{tikzpicture}[scale=0.5]     
      \node (T) at (0, -3) {};
      \node (R) at (0, 9) {};
      \node (S) at (0, 3) {};
      \node (i) at (0, 0) {};
      \node (j) at (0, 6) {};
  
      \node[isosceles triangle, 
      isosceles triangle apex angle=60, 
      draw, 
      rotate=270, 
      fill=white!50, 
      minimum size =24pt] (TT) at (T){};
      \node[isosceles triangle, 
      isosceles triangle apex angle=60, 
      draw, 
      rotate=270, 
      fill=white!50, 
      minimum size =24pt] (RT) at (R){};
      \node[isosceles triangle, 
      isosceles triangle apex angle=60, 
      draw, 
      rotate=270, 
      fill=white!50, 
      minimum size =24pt] (ST) at (S){};

      \draw[thick] (TT.west)--(i);
      \draw[thick] (i)--(ST.east);
      \draw[thick] (j)--(ST.west);
      \draw[thick] (RT.east)--(j);
  
  
      \draw[circle, fill=white] (i) circle [radius=24pt];
      \draw[circle, fill=white] (j) circle [radius=24pt];
      \draw[circle, fill=black] (RT.east) circle [radius=2pt];
      \draw[circle, fill=black] (TT.east) circle [radius=2pt];
      \draw[circle, fill=black] (ST.east) circle [radius=2pt];

      \node[cross=13] at (i) {};
      \node[cross=13] at (j) {};

  
      \node at (T) {\scalebox{0.8}{$T_0$}};
      \node at (S) {\scalebox{0.8}{$U_0$}};
      \node at (i) {\scalebox{0.8}{$v_i$}};
      \node at (j) {\scalebox{0.8}{$v_j$}};
      \node at (R) {\scalebox{0.8}{$V$}};
    \end{tikzpicture}}}
  \]
  with $V=Br(R;FU\cup FT_{k+1})$.\\
  Let us compute $T\star (S\star R)$:
  \[
    (T\star S)\star R=
    \vcenter{\hbox{\begin{tikzpicture}[scale=0.5]     
      \node (T) at (0, -3) {};
      \node (R) at (0, 9) {};
      \node (S) at (0, 3) {};
      \node (i) at (0, 0) {};
      \node (j) at (0, 6) {};
  
      \node[isosceles triangle, 
      isosceles triangle apex angle=60, 
      draw, 
      rotate=270, 
      fill=white!50, 
      minimum size =24pt] (TT) at (T){};
      \node[isosceles triangle, 
      isosceles triangle apex angle=60, 
      draw, 
      rotate=270, 
      fill=white!50, 
      minimum size =24pt] (RT) at (R){};
      \node[isosceles triangle, 
      isosceles triangle apex angle=60, 
      draw, 
      rotate=270, 
      fill=white!50, 
      minimum size =24pt] (ST) at (S){};

      \draw[thick] (TT.west)--(i);
      \draw[thick] (i)--(ST.east);
      \draw[thick] (j)--(ST.west);
      \draw[thick] (RT.east)--(j);
  
  
      \draw[circle, fill=white] (i) circle [radius=24pt];
      \draw[circle, fill=white] (j) circle [radius=24pt];
      \draw[circle, fill=black] (RT.east) circle [radius=2pt];
      \draw[circle, fill=black] (TT.east) circle [radius=2pt];
      \draw[circle, fill=black] (ST.east) circle [radius=2pt];

      \node[cross=13] at (i) {};
      \node[cross=13] at (j) {};

  
      \node at (T) {\scalebox{0.8}{$T_0$}};
      \node at (S) {\scalebox{0.8}{$U_0$}};
      \node at (i) {\scalebox{0.8}{$v_i$}};
      \node at (j) {\scalebox{0.8}{$v_j$}};
      \node at (R) {\scalebox{0.8}{$W$}};
    \end{tikzpicture}}}
  \]
  with $W=Br(Br(R;FS);\widetilde{FT_1})$ and $U_0=Br(S_0;FT_0)$, and $FT=FT_0\sqcup \widetilde{FT_1}$ for every
  possible such decomposition of $FT$ in (possibly empty) sub-forest. To conclude, one needs to remark that $V=W$. Indeed, in the definition of $W$, the forest $FS$ fall on $R$, then $\widetilde{FT_1}$ fall on the result; and in the definition of $V$, some trees of $\widetilde{FT_1}$ fall on some trees of $FS$ and the resulting forest fall on $R$. Those two operations give the same result.
\end{proof}

\begin{Definition}
  Let $\greg$ be the operad $(\G, \{\circ_i\})$. Let us note
  \[x=\vcenter{\hbox{\begin{tikzpicture}[scale=0.2]     
    \node (1) at (0, 0) {};
    \node (2) at (0, 3) {};


    \draw[thick] (1)--(2);


    \draw[circle, fill=white] (1) circle [radius=24pt];
    \draw[circle, fill=white] (2) circle [radius=24pt];
    
    \node at (1) {\scalebox{0.8}{$1$}};
    \node at (2) {\scalebox{0.8}{$2$}};

  \end{tikzpicture}}}\qquad 
  y=\vcenter{\hbox{\begin{tikzpicture}[scale=0.2]     
    \node (1) at (0, 0) {};
    \node (2) at (0, 3) {};


    \draw[thick] (1)--(2);


    \draw[circle, fill=white] (1) circle [radius=24pt];
    \draw[circle, fill=white] (2) circle [radius=24pt];
    
    \node at (2) {\scalebox{0.8}{$1$}};
    \node at (1) {\scalebox{0.8}{$2$}};

  \end{tikzpicture}}}\qquad
  g=\vcenter{\hbox{\begin{tikzpicture}[scale=0.2]     
    \node (1) at (-1.5, 3) {};
    \node (2) at (1.5, 3) {};
    \node (c) at (0, 0) {};


    \draw[thick] (c)--(1);
    \draw[thick] (c)--(2);


    \draw[circle, fill=black] (c) circle [radius=24pt];
    \draw[circle, fill=white] (1) circle [radius=24pt];
    \draw[circle, fill=white] (2) circle [radius=24pt];
    
    \node at (1) {\scalebox{0.8}{$1$}};
    \node at (2) {\scalebox{0.8}{$2$}};

  \end{tikzpicture}}}
  \]
  We know from Section \ref{seq:RT} that $x$ verify the pre-Lie relation. Let us introduce the following notation:
  \[
  x_n=
    \vcenter{\hbox{\begin{tikzpicture}[scale=0.2]     
    \node (T) at (0, 0) {};
    \node (S1) at (3, 3) {};
    \node (S2) at (0, 3) {};
    \node (S3) at (-3, 3) {};
      
    \draw[thick] (T)--(S1);
    \draw[thick] (T)--(S3);

    \node at (S2) {$\cdots$};

    \draw[circle, fill=white] (S1) circle [radius=24pt];
    \draw[circle, fill=white] (S3) circle [radius=24pt];
    \draw[circle, fill=white] (T) circle [radius=24pt];

  
    \node at (T) {\scalebox{0.8}{$1$}};
    \node at (S3) {\scalebox{0.8}{$2$}};
    \node at (S1) {\scalebox{0.8}{$n$}};
    \end{tikzpicture}}} 
    \qquad
  g_n=
    \vcenter{\hbox{\begin{tikzpicture}[scale=0.2]     
    \node (T) at (0, 0) {};
    \node (S1) at (3, 3) {};
    \node (S2) at (0, 3) {};
    \node (S3) at (-3, 3) {};
    
    \draw[thick] (T)--(S1);
    \draw[thick] (T)--(S3);
    \node at (S2) {$\cdots$};
    \draw[circle, fill=white] (S1) circle [radius=24pt];
    \draw[circle, fill=white] (S3) circle [radius=24pt];
    \draw[circle, fill=black] (T) circle [radius=24pt];
    
    \node at (S3) {\scalebox{0.8}{$1$}};
    \node at (S1) {\scalebox{0.8}{$n$}};
    \end{tikzpicture}}}
  \]
\end{Definition}

\begin{Example}
  Let us compute $x\circ_1 g$. Because of the renumbering that we have been ignoring, we have to compute
  $
  \vcenter{\hbox{\begin{tikzpicture}[scale=0.2]     
  \node (1) at (0, 0) {};
  \node (2) at (0, 3) {};


  \draw[thick] (1)--(2);


  \draw[circle, fill=white] (1) circle [radius=24pt];
  \draw[circle, fill=white] (2) circle [radius=24pt];
  
  \node at (1) {\scalebox{0.8}{$1$}};
  \node at (2) {\scalebox{0.8}{$3$}};

  \end{tikzpicture}}}\circ_1\vcenter{\hbox{\begin{tikzpicture}[scale=0.2]     
    \node (1) at (-1.5, 3) {};
    \node (2) at (1.5, 3) {};
    \node (c) at (0, 0) {};


    \draw[thick] (c)--(1);
    \draw[thick] (c)--(2);


    \draw[circle, fill=black] (c) circle [radius=24pt];
    \draw[circle, fill=white] (1) circle [radius=24pt];
    \draw[circle, fill=white] (2) circle [radius=24pt];
    
    \node at (1) {\scalebox{0.8}{$1$}};
    \node at (2) {\scalebox{0.8}{$2$}};

  \end{tikzpicture}}}
  $.
  \[
  \vcenter{\hbox{\begin{tikzpicture}[scale=0.2]     
    \node (1) at (0, 0) {};
    \node (2) at (0, 3) {};


    \draw[thick] (1)--(2);


    \draw[circle, fill=white] (1) circle [radius=24pt];
    \draw[circle, fill=white] (2) circle [radius=24pt];
    
    \node at (1) {\scalebox{0.8}{$1$}};
    \node at (2) {\scalebox{0.8}{$3$}};

  \end{tikzpicture}}}\circ_1\vcenter{\hbox{\begin{tikzpicture}[scale=0.2]     
    \node (1) at (-1.5, 3) {};
    \node (2) at (1.5, 3) {};
    \node (c) at (0, 0) {};


    \draw[thick] (c)--(1);
    \draw[thick] (c)--(2);


    \draw[circle, fill=black] (c) circle [radius=24pt];
    \draw[circle, fill=white] (1) circle [radius=24pt];
    \draw[circle, fill=white] (2) circle [radius=24pt];
    
    \node at (1) {\scalebox{0.8}{$1$}};
    \node at (2) {\scalebox{0.8}{$2$}};

  \end{tikzpicture}}}= 
  \vcenter{\hbox{\begin{tikzpicture}[scale=0.2]     
    \node (1) at (-1.5, 3) {};
    \node (2) at (1.5, 3) {};
    \node (3) at (-1.5, 6) {};
    \node (c) at (0, 0) {};


    \draw[thick] (c)--(1);
    \draw[thick] (c)--(2);
    \draw[thick] (1)--(3);


    \draw[circle, fill=black] (c) circle [radius=24pt];
    \draw[circle, fill=white] (1) circle [radius=24pt];
    \draw[circle, fill=white] (2) circle [radius=24pt];
    \draw[circle, fill=white] (3) circle [radius=24pt];
    
    \node at (1) {\scalebox{0.8}{$1$}};
    \node at (2) {\scalebox{0.8}{$2$}};
    \node at (3) {\scalebox{0.8}{$3$}};

  \end{tikzpicture}}}+
  \vcenter{\hbox{\begin{tikzpicture}[scale=0.2]     
    \node (1) at (-1.5, 3) {};
    \node (2) at (1.5, 3) {};
    \node (3) at (1.5, 6) {};
    \node (c) at (0, 0) {};


    \draw[thick] (c)--(1);
    \draw[thick] (c)--(2);
    \draw[thick] (2)--(3);


    \draw[circle, fill=black] (c) circle [radius=24pt];
    \draw[circle, fill=white] (1) circle [radius=24pt];
    \draw[circle, fill=white] (2) circle [radius=24pt];
    \draw[circle, fill=white] (3) circle [radius=24pt];
    
    \node at (1) {\scalebox{0.8}{$1$}};
    \node at (2) {\scalebox{0.8}{$2$}};
    \node at (3) {\scalebox{0.8}{$3$}};

  \end{tikzpicture}}}+
  \vcenter{\hbox{\begin{tikzpicture}[scale=0.2]     
    \node (1) at (-2, 3) {};
    \node (2) at (0, 3) {};
    \node (3) at (2, 3) {};
    \node (c) at (0, 0) {};


    \draw[thick] (c)--(1);
    \draw[thick] (c)--(2);
    \draw[thick] (c)--(3);


    \draw[circle, fill=black] (c) circle [radius=24pt];
    \draw[circle, fill=white] (1) circle [radius=24pt];
    \draw[circle, fill=white] (2) circle [radius=24pt];
    \draw[circle, fill=white] (3) circle [radius=24pt];
    
    \node at (1) {\scalebox{0.8}{$1$}};
    \node at (2) {\scalebox{0.8}{$2$}};
    \node at (3) {\scalebox{0.8}{$3$}};

  \end{tikzpicture}}}
  \]
  Hence, we have $x\circ_1 g- (g\circ_1 x).(2\;3) - g\circ_2 x = \vcenter{\hbox{\begin{tikzpicture}[scale=0.2]     
    \node (1) at (-2, 3) {};
    \node (2) at (0, 3) {};
    \node (3) at (2, 3) {};
    \node (c) at (0, 0) {};


    \draw[thick] (c)--(1);
    \draw[thick] (c)--(2);
    \draw[thick] (c)--(3);


    \draw[circle, fill=black] (c) circle [radius=24pt];
    \draw[circle, fill=white] (1) circle [radius=24pt];
    \draw[circle, fill=white] (2) circle [radius=24pt];
    \draw[circle, fill=white] (3) circle [radius=24pt];
    
    \node at (1) {\scalebox{0.8}{$1$}};
    \node at (2) {\scalebox{0.8}{$2$}};
    \node at (3) {\scalebox{0.8}{$3$}};

  \end{tikzpicture}}}$. One may recognize the Leibniz rule in the left-hand side of the equation. Moreover, since the right-hand side is symmetric, we have $x\circ_1 g- (g\circ_1 x).(2\;3) - g\circ_2 x = (x\circ_1 g- (g\circ_1 x).(2\;3) - g\circ_2 x).(2\;3)$. Let us call this relation the Greg relation.
\end{Example}

\begin{Remark}
  The element $g_3$ encodes the failure to verify the Leibniz's rule in the operad $\greg$, the same as the element $x_3$ encodes the failure to verify the associativity relation.
\end{Remark}

\begin{Proposition}
  The operad $\greg$ is binary.
\end{Proposition}

\begin{proof}
  Let $\P(x, y, g)$ be the suboperad of $\greg$ generated by $x$, $y$ and $g$. We have to show that $\P(x, y, g)=\greg$. Let us prove it by induction on the arity.\\
  \emph{Base case:} By definition $\P(x, y, g)(2)=\greg(2)$.\\
  \emph{Induction step:} Let $n\geq 2$ and suppose that $\P(x, y, g)(k)=\greg(k)$ for all $k\leq n$. We have to show that $\P(x, y, g)(n+1)=\greg(n+1)$.\\
  Computing $x\circ_1x_n$ and $x\circ_1c_n$ shows that $x_{n+1}\in \P(x, y, g)(n+1)$ and $c_{n+1}\in \P(x, y, g)(n+1)$. Since we can obtain any rooted Greg trees by inductively composing corollas in the leaves of smaller trees, we have $\greg(n+1)= \P(x, y, g)(n+1)$.\\
  Hence by induction, $\P(x, y, g)=\greg$.
\end{proof}

We want to prove that $\greg$ is quadratic to get a quadratic presentation. In order to do so, we will introduce a quadratic operad $\greg'$, show that $\greg'$ is Koszul and use the information given on its dimensions of components to show that $\greg'$ is isomorphic to $\greg$.

\begin{Definition}
  Let $\greg'$ be the quotient of the free operad generated by $\tx $ and $\tg $ such that $\tx $ has no symmetry and $\tg.(1\;2) =\tg $, by the following relations:
  \begin{itemize}
    \item $(\tx  \circ_1\tx  -\tx  \circ_2\tx ) - (\tx  \circ_1\tx  -\tx  \circ_2\tx ).(2\;3)$
    \item $(\tx \circ_1 \tg - (\tg \circ_1 \tx ).(2\;3) - \tg \circ_2 \tx ) - 
    (\tx \circ_1 \tg - (\tg \circ_1 \tx ).(2\;3) - \tg \circ_2 \tx ).(2\;3)$
  \end{itemize}
  The first relation is the pre-Lie relation and the second one is the Greg relation. Since $\greg$ is binary and those relations are satisfied in $\greg$, we have a surjective morphism of operads $\greg'\twoheadrightarrow \greg$.
\end{Definition}

\begin{Remark}
  From this definition and using the formalism of shuffle operads, it would already be possible to prove that $\greg'$ is Koszul. However, the dimensions of components of its Koszul dual are much simpler, hence we will compute and work with the Koszul dual of $\greg'$.
\end{Remark}

\begin{Definition}
  Let $(\greg')^!$ be the quotient of the free operad generated by $x_*$ and $g_*$ such that $x_*$ has no symmetry and $g_*.(1\;2)=-g_*$, by the following relations:
  $$x_*\circ_1x_*-x_*\circ_2x_* \qquad ; \qquad x_*\circ_1x_*-(x_*\circ_1x_*).(2\;3)$$
  $$x_*\circ_1g_*-g_*\circ_2x_* \qquad ; \qquad x_*\circ_1g_*-(g_*\circ_1x_*).(2\;3)$$
  $$x_*\circ_1g_*+(x_*\circ_1g_*).(1\;2\;3)+(x_*\circ_1g_*).(1\;3\;2)$$
  $$x_*\circ_2g_*  \qquad ; \qquad g_*\circ_1g_*$$

  This is the Koszul dual of $\greg'$. The explicit method to compute a presentation of the Koszul dual of a binary operad is detailed in \cite[Subsection 7.6]{AlgOp}. The two first relation are associativity and permutativity. This operad is not a set operad because of the last three relations.
\end{Definition}

Let us give an example of $(\greg')^!$-algebra that allows us to show that $\dim((\greg')^!(4))\geq 7$.\\

\begin{Definition}\label{Definition:walg}
  Let $\chi$ be a finite alphabet and $\mathbf{W}(\chi)$ be the span of finite words on $\chi$ with the following extra decorations: either one letter is pointed with a dot, or there is an arrow from one letter to another.
  Let $W(\chi)$ be the quotient of $\mathbf{W}(\chi)$ by the following relations, letters commute with each other (the dot and the arrow follow the letter), reverting the arrow changes the sign and $$\overset{\curvearrowright}{ab}cv=\overset{\curvearrowright}{cb}av+\overset{\curvearrowright}{ac}bv$$ for any $a, b, c\in \chi$ and $v$ a finite word. Because the letters commute, we can write the elements of $W(\chi)$ with the pointed letter (or arrowed letters) at the start.\\
  Let the $\ccircled{x}$ and $\ccircled{g}$ products on $W(\chi)$ defined by:
  \begin{itemize}
    \item $\dot{a}v\ccircled{x}\dot{b}w=\dot{a}vbw$
    \item $\overset{\curvearrowright}{ab}v\ccircled{x}\dot{c}w=\overset{\curvearrowright}{ab}vcw$
    \item $\dot{a}v\ccircled{g}\dot{b}w=\overset{\curvearrowright}{ab}vw$
  \end{itemize}
  All other cases give $0$.
\end{Definition} 

\begin{Proposition}
  The algebra $(W(\chi), \ccircled{x}, \ccircled{g})$ is a $(\greg')^!$-algebra generated by $\chi$.
\end{Proposition}

\begin{proof}
  Indeed, $\dot{a}v=\dot{a}\ccircled{x}w$ with $w$ the word $v$ with a dot on a letter (let us say the first one for example) and 
  $$\overset{\curvearrowright}{ab}v=(\dot{a}\ccircled{g}\dot{b})\ccircled{x}w$$ 
  so $(W(\chi), \ccircled{x}, \ccircled{g})$ it is generated by $\chi$. The product $\ccircled{g}$ is skew-symmetric, indeed $$\dot{a}\ccircled{g}\dot{b}=\overset{\curvearrowright}{ab}=\overset{\curvearrowleft}{ba}=-\overset{\curvearrowright}{ba}=-\dot{b}\ccircled{g}\dot{a}$$ The $7$ relations of $(\greg')^!$ are easily checked.
\end{proof}

\begin{Proposition}
  We have $\dim((\greg')^!(4))\geq 7$.
\end{Proposition}

\begin{proof}
  Let $A=W(\{a, b, c, d\})$. Let us compute the dimension of $\Mult(A(4))$ the multilinear part of $A(4)$. Let us write words of $\Mult(A(4))$ such that the arrow always starts from $a$ and go to the second letter, and aside from that, the letters are in the lexicographic order.\\
  The rewriting rule $\overset{\curvearrowright}{\beta\gamma}\alpha\mapsto \overset{\curvearrowright}{\alpha\gamma}\beta-\overset{\curvearrowright}{\alpha\beta}\gamma$ is confluent. Indeed: 
  $$\overset{\curvearrowright}{cd}ab=\overset{\curvearrowright}{bd}ac-\overset{\curvearrowright}{bc}ad=  (\overset{\curvearrowright}{ad}bc-\overset{\curvearrowright}{ab}cd)-(\overset{\curvearrowright}{ac}bd-\overset{\curvearrowright}{ab}cd)=\overset{\curvearrowright}{ad}bc-\overset{\curvearrowright}{ac}bd$$
  Hence, $\Mult(A(4))$ is spanned by the words $\dot{a}bcd$, $a\dot{b}cd$, $ab\dot{c}d$, $abc\dot{d}$, $\overset{\curvearrowright}{ab}cd$, $\overset{\curvearrowright}{ac}bd$ and $\overset{\curvearrowright}{ad}bc$. Since $A$ is a $(\greg')^!$-algebra on $4$ generators, $\dim((\greg')^!(4))\geq\dim(\Mult(A(4)))=7$.
\end{proof}

\begin{Remark}
  The algebra $(W(\chi), \ccircled{x}, \ccircled{g})$ is in fact the free $(\greg')^!$-algebra generated by $\chi$. 
\end{Remark}

Let us use the formalism of shuffle operads to write down a Gröbner basis of $(\greg')^!$. Operadic monomials will be written as shuffle trees. However, one needs not to confuse them with Greg trees. For further information on shuffle operads, see \cite{AlgComp}. A quick introduction on shuffle trees has been given in the first section see \ref{rec:sh}. Writing the relations of $(\greg')^!$ using shuffle trees is a good exercise to familiarize ourselves with shuffle trees, and to be careful not to confuse them with the species of rooted trees or rooted Greg trees. Since the actions of the symmetric groups are disposed of when working with shuffle operads, let us note $y_*=x_*.(1\;2)$. We get Figure \ref{fg:StRel1}.
\paragraph{}
From those computations, the only missing ingredient to get a Gröbner basis is a monomial order. We will consider the three following orders to get terminating rewriting systems. Those orders are defined in \cite{AlgComp} for any presentation of shuffle operads.
\begin{itemize}
  \item The degree-lexicographic permutation order with $x_*>y_*>g_*$ gives Figure \ref{fg:StGB1}.
  \item The permutation reverse-degree-lexicographic order with $g_*>y_*>x_*$ gives Figure \ref{fg:StGB2}.
  \item The reverse-degree-lexicographic permutation order with $x_*>y_*>g_*$ gives Figure \ref{fg:StGB3}.
\end{itemize}

\begin{Proposition}\label{Proposition:GB}
  The rewriting systems displayed in  Figure \ref{fg:StGB1}, Figure \ref{fg:StGB2} and Figure \ref{fg:StGB3} are confluent. They are in fact Gröbner bases.
\end{Proposition}

\begin{proof}
  To prove this fact one has two choices either checking the confluence of the critical monomials ($294$ cases to check for Figure \ref{fg:StGB1}), or remarking that there are $7$ normal forms in arity $4$ for each rewriting system and because $\dim((\greg')^!(4))\geq 7$, no new relation can appear. Since we have a monomial order and the rewriting rules are quadratic in a binary operad, checking arity $4$ is enough.
\end{proof}

\begin{Proposition}
  We have that $\dim((\greg')^!(n))=2n-1$ for all $n\geq 1$, hence its exponential generating series is $(2t-1)\exp(t)+1$.
\end{Proposition}

\begin{proof}
  It suffices to count the normal forms of a rewriting system for example Figure \ref{fg:StGB1}. Let $n\geq 2$ and count the number of normal forms in arity $n$. Those are right combs with at most one $g_*$, with all the $x_*$ above the $g_*$ and the $y_*$, and all the $y_*$ below the $g_*$ and the $x_*$. Hence, the normal forms are determined by the number of occurrences $x_*$ and $g_*$, and have either zero or one occurrence $g_*$. If there is no $g_*$, then one can have from $0$ to $n-1$ occurrences of $x_*$. If there is one $g_*$, then one can have from $0$ to $n-2$ occurrences of $x_*$. Hence, the number of normal forms in arity $n$ is $2n-1$.
\end{proof}

\begin{Theorem}
  The operad $\greg'$ is Koszul.
\end{Theorem}

\begin{proof}
  We have a quadratic Gröbner basis for $(\greg')^!$, hence $(\greg')^!$ is Koszul, hence $\greg'$ is Koszul.
\end{proof}

\begin{Theorem}  
  The exponential generating series of the operad $\greg'$ is the inverse under composition of $(2t+1)\exp(-t)-1$. Hence, $\greg'$ is isomorphic to $\greg$.
\end{Theorem}

\begin{proof}
  We know that a Koszul operad $\P$ satisfies $f_\P(f_{\P^!}(-t))=-t$. Hence, the exponential generating series of the operad $\greg'$ is the inverse under composition of $(2t+1)\exp(-t)-1$. Since we have a surjective morphism from $\greg'$ to $\greg$ and that they have the same exponential generating series, the morphism is an isomorphism.
\end{proof}

\begin{Corollary}\label{Theorem:m1}
  The operad $\greg$ is binary, quadratic and Koszul.
\end{Corollary}

\section{Generalization of the Greg operad}\label{seq:GGO}

We have studied the operad $\greg$, however we have not yet relate this operad to $\bigvee_\Lie^2\PL$. In fact, we shall now establish a much more general result about the operad $\bigvee_\Lie^{n+1}\PL$.

\begin{Proposition}
  The operad $\bigvee_\Lie^{n+1}\PL$ is isomorphic to the operad $\F(x, c_1, \dots, c_n)/\langle \mathcal{R}\rangle$, with $x$ without symmetries, $c_k.(1\;2)=c_k$ and $\mathcal{R}$ the relations:
  \begin{equation}\tag{pre-Lie}\label{eq:pl}
    (x \circ_1x -x \circ_2x) - (x \circ_1x -x \circ_2x).(2\;3)
  \end{equation}
  \begin{multline}\tag{diff pre-Lie}\label{eq:vpl}
    (x \circ_1 c_k - (c_k \circ_1 x ).(2\;3) - c_k \circ_2 x ) - 
  (x \circ_1 c_k - (c_k \circ_1 x ).(2\;3) - c_k \circ_2 x ).(2\;3)\\
   +  \sum_{i, j\mid \max(i, j)=k}(c_i\circ_1 c_j - (c_i\circ_1 c_j).(2\;3))
  \end{multline}
\end{Proposition}

\begin{proof}
  We already know the following presentation: $\bigvee_\Lie^{n+1}\PL\simeq\F(x_1, \dots, x_{n+1})/\langle \mathcal{R}' \rangle$ with $x_k$ without symmetries and $\mathcal{R}'$ the relations:
  \begin{equation}
    \tag{pre-Lie k}\label{eq:pli}
    (x_k \circ_1x_k -x_k \circ_2x_k) - (x_k \circ_1x_k -x_k \circ_2x_k).(2\;3)
  \end{equation}
  \begin{equation}
    \tag{share}\label{eq:share}
    (x_k-x_{k+1})-(x_k-x_{k+1}).(1\;2)
  \end{equation}
  Let $c_k=x_k-x_{k+1}$, then $c_k=c_k.(1\;2)$ is equivalent to Relation (\ref{eq:share}). Let $x=x_1$. We have that $x_{k+1}=x+\sum_{i=1}^{k}c_i$. Hence, $\mathcal{R}'$ is equivalent to:
  \begin{multline*}
    (x \circ_1x -x \circ_2x) - (x \circ_1x -x \circ_2x).(2\;3) + \\
    \sum_{i=1}^{k}((x \circ_1c_i -x \circ_2c_i) - 
    (x \circ_1c_i -x \circ_2c_i).(2\;3) + (c_i \circ_1x -c_i \circ_2x) - 
    (c_i \circ_1x -c_i \circ_2x).(2\;3)) + \\
    \sum_{i=1}^{k}\sum_{j=1}^{k}((c_i \circ_1c_j -c_i \circ_2c_j) - 
    (c_i \circ_1c_j -c_i \circ_2c_j).(2\;3)), 
  \end{multline*}
  which is equal to:
  \begin{multline*}
    (x \circ_1x -x \circ_2x) - (x \circ_1x -x \circ_2x).(2\;3) + \\
    \sum_{i=1}^{k}(x \circ_1c_i - 
    (x \circ_1c_i).(2\;3) + (c_i \circ_1x -c_i \circ_2x) - 
    (c_i \circ_1x -c_i \circ_2x).(2\;3)) + \\
    \sum_{i=1}^{k}\sum_{j=1}^{k}(c_i \circ_1c_j - 
    (c_i \circ_1c_j).(2\;3))
  \end{multline*}
  Finally, if we subtract consecutive relations; we obtain
  \begin{multline*}
    x \circ_1c_k - 
    (x \circ_1c_k).(2\;3) + (c_i \circ_1x -c_k \circ_2x) - 
    (c_k \circ_1x -c_k \circ_2x).(2\;3) +\\
    \sum_{i, j\mid\max(i, j)=k}(c_i \circ_1c_j - 
    (c_i \circ_1c_j).(2\;3)), 
  \end{multline*}
  which is the intended relation.
\end{proof}

This quadratic presentation very much look like the presentation of the operad $\greg$. The operad $\bigvee_\Lie^2\PL$ is not isomorphic to $\greg$, however one may wonder if the operad $\bigvee_\Lie^2\PL$ is a deformation of $\greg$. We shall now show that it is indeed the case.

\paragraph{}
Let $C=(V, \Delta)$ with $V$ be a vector space of finite dimension $n$ and $\Delta$ a coassociative cocommutative coproduct on $V$.

\begin{Definition}
  The vector space of \emph{rooted Greg trees over $V$} is $\G_k^V(m)=\bigotimes_{\tau\in G_k(m)}V^{\otimes BV(\tau)}$. It has a basis of rooted Greg trees whose black vertices are labeled by a basis of $V$.\\ 
  Let $\G^V(m)=\bigoplus_k\G_k^V(m)$ and $\G^V=\bigoplus_m\G^V(m)$.
\end{Definition}

\begin{Definition}
  Let us define the deformed fall product $\star^\Delta$ on $\G^V$. Let $S$ and $T$ be two rooted Greg trees over $V$, and reuse the notation of Section \ref{seq:RT}.
  For $v$ a vertex of $S$, the forest $FS=\{S_1, \dots, S_k\}$ the forest of the children of $v$, $B$ the rooted tree below $v$ and $c^v$ the label of $v$. Let us write:
  \[
    \vcenter{\hbox{\begin{tikzpicture}[scale=0.5]     
      \node (S) at (0, 0) {};

      \node[isosceles triangle, 
      isosceles triangle apex angle=60, 
      draw, 
      rotate=270, 
      fill=white!50, 
      minimum size =24pt] (ST) at (S){};



      \draw[circle, fill=black] (ST.east) circle [radius=2pt];

      \node at (S) {\scalebox{0.8}{$S$}};
    \end{tikzpicture}}}
    \qquad
    =
    \qquad
    \vcenter{\hbox{\begin{tikzpicture}[scale=0.5]     
      \node (lv) at (0, 0) {};
      \node (SD) at (0, -2.7) {};
      \node (F) at (0, 3) {};

      \node[isosceles triangle, 
      isosceles triangle apex angle=60, 
      draw, 
      rotate=270, 
      fill=white!50, 
      minimum size =24pt] (SDT) at (SD){};
      \node[isosceles triangle, 
      isosceles triangle apex angle=60, 
      draw, 
      rotate=270, 
      fill=white!50, 
      minimum size =24pt] (FT) at (F){};

      \draw[thick, double] (FT.east)--(lv);
      \draw[thick] (lv)--(SDT.west);


      \draw[circle, fill=white] (lv) circle [radius=24pt];
      \draw[circle, fill=black] (FT.east) circle [radius=2pt];
      \draw[circle, fill=black] (SDT.east) circle [radius=2pt];

      \node at (lv) {\scalebox{1}{$c^v$}};

      \node at (F) {\scalebox{0.8}{$FS$}};
      \node at (SD) {\scalebox{0.8}{$B$}};
    \end{tikzpicture}}}
  \]
  For $v$ a black vertex and $c^v$ its label, let us write $c_{(1)}^v\otimes c_{(2)}^v=\Delta(c^v)$ using the Sweedler notation. Let us define the product $\star^\Delta$ by:
  \[
  S\star^\Delta T = \resum[1]_{v\in V(S)}
  \vcenter{\hbox{\begin{tikzpicture}[scale=0.5]     
    \node (lv) at (0, 0) {};
    \node (SD) at (0, -3) {};
    \node (T) at (1.5, 3) {};
    \node (FS) at (-1.5, 3) {};

    \node[isosceles triangle, 
    isosceles triangle apex angle=60, 
    draw, 
    rotate=270, 
    fill=white!50, 
    minimum size =24pt] (SDT) at (SD){};
    \node[isosceles triangle, 
    isosceles triangle apex angle=60, 
    draw, 
    rotate=270, 
    fill=white!50, 
    minimum size =24pt] (TT) at (T){};
    \node[isosceles triangle, 
    isosceles triangle apex angle=60, 
    draw, 
    rotate=270, 
    fill=white!50, 
    minimum size =24pt] (FST) at (FS){};    
    \draw[thick] (TT.east)--(lv);
    \draw[thick, double] (FST.east)--(lv);
    \draw[thick] (lv)--(SDT.west);


    \draw[circle, fill=white] (lv) circle [radius=24pt];
    \draw[circle, fill=black] (TT.east) circle [radius=2pt];
    \draw[circle, fill=black] (FST.east) circle [radius=2pt];
    \draw[circle, fill=black] (SDT.east) circle [radius=2pt];

    \node at (lv) {\scalebox{1}{$c^v$}};

    \node at (T) {\scalebox{0.8}{$T$}};
    \node at (FS) {\scalebox{0.8}{$FS$}};
    \node at (SD) {\scalebox{0.8}{$B$}};
  \end{tikzpicture}}}\quad + \resum[1]_{v\in BV(S)}\;\resum[1]_{FS_{(1)}\sqcup FS_{(2)}=FS}
  \vcenter{\hbox{\begin{tikzpicture}[scale=0.5]     
    \node (lv0) at (-1.5, -3) {};
    \node (lv) at (0, 0) {};
    \node (SD) at (-1.5, -6) {};
    \node (T) at (1.5, 3) {};
    \node (FS) at (-1.5, 3) {};
    \node (FS2) at (-3, 0) {};

    \node[isosceles triangle, 
    isosceles triangle apex angle=60, 
    draw, 
    rotate=270, 
    fill=white!50, 
    minimum size =24pt] (SDT) at (SD){};
    \node[isosceles triangle, 
    isosceles triangle apex angle=60, 
    draw, 
    rotate=270, 
    fill=white!50, 
    minimum size =24pt] (TT) at (T){};
    \node[isosceles triangle, 
    isosceles triangle apex angle=60, 
    draw, 
    rotate=270, 
    fill=white!50, 
    minimum size =28pt] (FST) at (FS){};  
    \node[isosceles triangle, 
    isosceles triangle apex angle=60, 
    draw, 
    rotate=270, 
    fill=white!50, 
    minimum size =28pt] (FS2T) at (FS2){};    
    \draw[thick] (TT.east)--(lv);
    \draw[thick, double] (FST.east)--(lv);
    \draw[thick, double] (FS2T.east)--(lv0);
    \draw[thick] (lv)--(lv0);
    \draw[thick] (lv0)--(SDT.west);


    \draw[circle, fill=white] (lv0) circle [radius=24pt];
    \draw[circle, fill=white] (lv) circle [radius=24pt];
    \draw[circle, fill=black] (TT.east) circle [radius=2pt];
    \draw[circle, fill=black] (FST.east) circle [radius=2pt];
    \draw[circle, fill=black] (FS2T.east) circle [radius=2pt];
    \draw[circle, fill=black] (SDT.east) circle [radius=2pt];

    \node at (lv0) {\scalebox{1}{$c_{(1)}^v$}};
    \node at (lv) {\scalebox{1}{$c_{(2)}^v$}};

    \node at (T) {\scalebox{0.8}{$T$}};
    \node at (FS) {\scalebox{0.8}{$FS_{(2)}$}};
    \node at (FS2) {\scalebox{0.8}{$FS_{(1)}$}};
    \node at (SD) {\scalebox{0.8}{$B$}};
  \end{tikzpicture}}}
  \]
  For the sake of readability, let us write:
  \[
    S\star^\Delta T  = \qquad
    \vcenter{\hbox{\begin{tikzpicture}[scale=0.5]     
      \node (lv) at (0, 0) {};
      \node (T) at (0, 3) {};

      \node[isosceles triangle, 
      isosceles triangle apex angle=60, 
      draw, 
      rotate=270, 
      fill=white!50, 
      minimum size =24pt] (TT) at (T){};
      
      \draw[thick] (TT.east)--(lv);


      \draw[circle, fill=white] (lv) circle [radius=24pt];
      \draw[circle, fill=black] (TT.east) circle [radius=2pt];
      
      \node at (lv) {\scalebox{1}{$c^v$}};

      \node at (T) {\scalebox{0.8}{$T$}};
    \end{tikzpicture}}} \qquad + \qquad
    \vcenter{\hbox{\begin{tikzpicture}[scale=0.5]     
      \node (lv0) at (0, -3) {};
      \node (lv) at (0, 0) {};
      \node (T) at (0, 3) {};

      \node[isosceles triangle, 
      isosceles triangle apex angle=60, 
      draw, 
      rotate=270, 
      fill=white!50, 
      minimum size =24pt] (TT) at (T){};
      
      \draw[thick] (TT.east)--(lv);
      \draw[thick] (lv)--(lv0);


      \draw[circle, fill=white] (lv0) circle [radius=24pt];
      \draw[circle, fill=white] (lv) circle [radius=24pt];
      \draw[circle, fill=black] (TT.east) circle [radius=2pt];
      
      \node at (lv0) {\scalebox{1}{$c_{(1)}^v$}};
      \node at (lv) {\scalebox{1}{$c_{(2)}^v$}};

      \node at (T) {\scalebox{0.8}{$T$}};
    \end{tikzpicture}}}  
  \]
\end{Definition}

\begin{Proposition}
  The deformed fall product $\star^\Delta$ is pre-Lie.
\end{Proposition}

\begin{proof}
  The proof is the tedious computation of $(R\star^\Delta S) \star^\Delta T-R\star^\Delta (S \star^\Delta T)$. The computation of $(R\star^\Delta S) \star^\Delta T$ is done in Figure \ref{fg:StComp1}; $r$, $r'$ and $s$ are labels of vertices of $R$ and $S$ respectively. The boxed terms are the terms of $R\star^\Delta (S \star^\Delta T)$. Using the coassociativity and cocommutativity of $\Delta$, we get that $(R\star^\Delta S) \star^\Delta T-R\star^\Delta (S \star^\Delta T)$ is symmetric in $S$ and $T$. Hence, the deformed fall product $\star^\Delta$ is pre-Lie.
\end{proof}

\begin{Remark}
  One may remark that cocommutativity is stronger than the needed condition, indeed the weaker needed  condition is that $r_{(1)}\otimes r_{(2)}\otimes r_{(3)}=r_{(1)}\otimes r_{(3)}\otimes r_{(2)}$ using Sweedler notation, which is known as the copermutativity property.
\end{Remark}

The symmetric brace product $Br^\Delta$ associated to $\star^\Delta$ is also defined by the following formula:
$$Br^\Delta(S;T_1, \dots, T_{n+1})=Br^\Delta(S;T_1, \dots, T_n)\star^\Delta T_{n+1}- \sum_{i=1}^nBr^\Delta(S;T_1, \dots, T_i\star^\Delta T_{n+1}, \dots, T_n)$$
Same as $Br$, $Br^\Delta$ is symmetric in the $T_i$'s.

\begin{Definition}
  The partial compositions $\circ_i^\Delta$ are defined the same way as in Definition \ref{Definition:inf} by:
  \[
    T\circ_i^\Delta S = 
    \vcenter{\hbox{\begin{tikzpicture}[scale=0.5]     
      \node (T) at (0, 0) {};
      \node (i) at (0, 3) {};
      \node (S) at (0, 6) {};

      \node[isosceles triangle, 
      isosceles triangle apex angle=60, 
      draw, 
      rotate=270, 
      fill=white!50, 
      minimum size =24pt] (TT) at (T){};
      \node[isosceles triangle, 
      isosceles triangle apex angle=60, 
      draw, 
      rotate=270, 
      fill=white!50, 
      minimum size =24pt] (ST) at (S){}; 

      \draw[thick] (TT.west)--(ST.east);


      \draw[circle, fill=black] (TT.east) circle [radius=2pt];
      \draw[circle, fill=black] (ST.east) circle [radius=2pt];
      \draw[circle, fill=white] (i) circle [radius=24pt];

      \node[cross=13pt] at (i) {};
      
      \node at (T) {\scalebox{0.8}{$B$}};
      \node at (i) {\scalebox{0.8}{$v$}};
      \node at (S) {\scalebox{0.8}{$U$}};

    \end{tikzpicture}}}
  \]
  With $U=Br^\Delta(S;FT)$.
\end{Definition}

\begin{Proposition}
  The partial compositions $\circ_i^\Delta$ satisfy the sequential composition and parallel composition axioms.
\end{Proposition}

The proof is the same as the one of Proposition \ref{Proposition:inf}.

\begin{Definition}
  With $C=(V, \Delta)$, let $\greg^C$ be the operad $(\G^V, \{\circ_i^\Delta\})$. Let $(e^1, \dots, e^n)$ be a basis of $V$ and:
  \[
  x_n=
    \vcenter{\hbox{\begin{tikzpicture}[scale=0.3]     
    \node (T) at (0, 0) {};
    \node (S1) at (3, 3) {};
    \node (S2) at (0, 3) {};
    \node (S3) at (-3, 3) {};
      
    \draw[thick] (T)--(S1);
    \draw[thick] (T)--(S3);

    \node at (S2) {$\cdots$};

    \draw[circle, fill=white] (S1) circle [radius=24pt];
    \draw[circle, fill=white] (S3) circle [radius=24pt];
    \draw[circle, fill=white] (T) circle [radius=24pt];

  
    \node at (T) {\scalebox{0.8}{$1$}};
    \node at (S3) {\scalebox{0.8}{$2$}};
    \node at (S1) {\scalebox{0.8}{$n$}};
    \end{tikzpicture}}} 
    \qquad
  g_n^k=
    \vcenter{\hbox{\begin{tikzpicture}[scale=0.3]     
    \node (T) at (0, 0) {};
    \node (S1) at (3, 3) {};
    \node (S2) at (0, 3) {};
    \node (S3) at (-3, 3) {};
    
    \draw[thick] (T)--(S1);
    \draw[thick] (T)--(S3);
    \node at (S2) {$\cdots$};
    \draw[circle, fill=white] (S1) circle [radius=24pt];
    \draw[circle, fill=white] (S3) circle [radius=24pt];
    \draw[circle, fill=black] (T) circle [radius=28pt];
    
    \node at (S3) {\scalebox{0.8}{$1$}};
    \node at (S1) {\scalebox{0.8}{$n$}};
    \node at (T) {\scalebox{1}{$\textcolor{white}{e^k}$}};
    \end{tikzpicture}}}
  \]
  Let $x=x_2$ and $g^k=g_2^k$.
\end{Definition}

\begin{Proposition}
  The operad $\greg^{C}$ is binary and satisfy the following relations:
  \begin{equation}\tag{pre-Lie}
    (x \circ_1x -x \circ_2x) - (x \circ_1x -x \circ_2x).(2\;3)
  \end{equation}
  \begin{multline}\tag{greg $\Delta$}\label{eq:gregDelta}
    (x \circ_1 g^k - (g^k \circ_1 x ).(2\;3) - g^k \circ_2 x ) - 
    (x \circ_1 g^k - (g^k \circ_1 x ).(2\;3) - g^k \circ_2 x ).(2\;3)\\
    +  (g^k_{(1)}\circ_1 g^k_{(2)} - (g^k_{(1)}\circ_1 g^k_{(2)}).(2\;3))  
  \end{multline}
  With $g^k_{(1)}$ and $g^k_{(2)}$ defined by $\Delta$ under the identification of $V$ with the span of the generators $g^k$.
\end{Proposition}

\begin{proof}
  Let us compute $x \circ_1 g^k$:
  \[
  \vcenter{\hbox{\begin{tikzpicture}[scale=0.3]     
    \node (1) at (0, 0) {};
    \node (2) at (0, 3) {};
    
    
    \draw[thick] (1)--(2);
    
    
    \draw[circle, fill=white] (1) circle [radius=24pt];
    \draw[circle, fill=white] (2) circle [radius=24pt];
    
    \node at (1) {\scalebox{0.8}{$1$}};
    \node at (2) {\scalebox{0.8}{$3$}};

  \end{tikzpicture}}}\circ_1\vcenter{\hbox{\begin{tikzpicture}[scale=0.3]     
    \node (1) at (-1.5, 3) {};
    \node (2) at (1.5, 3) {};
    \node (c) at (0, 0) {};


    \draw[thick] (c)--(1);
    \draw[thick] (c)--(2);


    \draw[circle, fill=black] (c) circle [radius=28pt];
    \draw[circle, fill=white] (1) circle [radius=24pt];
    \draw[circle, fill=white] (2) circle [radius=24pt];
    
    \node at (1) {\scalebox{0.8}{$1$}};
    \node at (2) {\scalebox{0.8}{$2$}};
    \node at (c) {\scalebox{1}{$\textcolor{white}{e^k}$}};

  \end{tikzpicture}}}= 
  \vcenter{\hbox{\begin{tikzpicture}[scale=0.3]     
    \node (1) at (-1.5, 3) {};
    \node (2) at (1.5, 3) {};
    \node (3) at (-1.5, 6) {};
    \node (c) at (0, 0) {};


    \draw[thick] (c)--(1);
    \draw[thick] (c)--(2);
    \draw[thick] (1)--(3);


    \draw[circle, fill=black] (c) circle [radius=28pt];
    \draw[circle, fill=white] (1) circle [radius=24pt];
    \draw[circle, fill=white] (2) circle [radius=24pt];
    \draw[circle, fill=white] (3) circle [radius=24pt];
    
    \node at (1) {\scalebox{0.8}{$1$}};
    \node at (2) {\scalebox{0.8}{$2$}};
    \node at (3) {\scalebox{0.8}{$3$}};
    \node at (c) {\scalebox{1}{$\textcolor{white}{e^k}$}};

  \end{tikzpicture}}}+
  \vcenter{\hbox{\begin{tikzpicture}[scale=0.3]     
    \node (1) at (-1.5, 3) {};
    \node (2) at (1.5, 3) {};
    \node (3) at (1.5, 6) {};
    \node (c) at (0, 0) {};


    \draw[thick] (c)--(1);
    \draw[thick] (c)--(2);
    \draw[thick] (2)--(3);


    \draw[circle, fill=black] (c) circle [radius=28pt];
    \draw[circle, fill=white] (1) circle [radius=24pt];
    \draw[circle, fill=white] (2) circle [radius=24pt];
    \draw[circle, fill=white] (3) circle [radius=24pt];
    
    \node at (1) {\scalebox{0.8}{$1$}};
    \node at (2) {\scalebox{0.8}{$2$}};
    \node at (3) {\scalebox{0.8}{$3$}};
    \node at (c) {\scalebox{1}{$\textcolor{white}{e^k}$}};

  \end{tikzpicture}}}+
  \vcenter{\hbox{\begin{tikzpicture}[scale=0.3]     
    \node (1) at (-2, 3) {};
    \node (2) at (0, 3) {};
    \node (3) at (2, 3) {};
    \node (c) at (0, 0) {};


    \draw[thick] (c)--(1);
    \draw[thick] (c)--(2);
    \draw[thick] (c)--(3);


    \draw[circle, fill=black] (c) circle [radius=28pt];
    \draw[circle, fill=white] (1) circle [radius=24pt];
    \draw[circle, fill=white] (2) circle [radius=24pt];
    \draw[circle, fill=white] (3) circle [radius=24pt];
    
    \node at (1) {\scalebox{0.8}{$1$}};
    \node at (2) {\scalebox{0.8}{$2$}};
    \node at (3) {\scalebox{0.8}{$3$}};
    \node at (c) {\scalebox{1}{$\textcolor{white}{e^k}$}};

  \end{tikzpicture}}}+
  \vcenter{\hbox{\begin{tikzpicture}[scale=0.3]     
    \node (1) at (-1.5, 3) {};
    \node (2) at (0, 6) {};
    \node (3) at (3, 6) {};
    \node (c) at (0, 0) {};
    \node (g) at (1.5, 3) {};


    \draw[thick] (c)--(1);
    \draw[thick] (c)--(g);
    \draw[thick] (g)--(2);
    \draw[thick] (g)--(3);


    \draw[circle, fill=black] (c) circle [radius=32pt];
    \draw[circle, fill=black] (g) circle [radius=32pt];
    \draw[circle, fill=white] (1) circle [radius=24pt];
    \draw[circle, fill=white] (2) circle [radius=24pt];
    \draw[circle, fill=white] (3) circle [radius=24pt];
    
    \node at (1) {\scalebox{0.8}{$1$}};
    \node at (2) {\scalebox{0.8}{$2$}};
    \node at (3) {\scalebox{0.8}{$3$}};
    \node at (c) {\scalebox{1}{$\textcolor{white}{e_{(1)}^k}$}};
    \node at (g) {\scalebox{1}{$\textcolor{white}{e_{(2)}^k}$}};

  \end{tikzpicture}}}+
  \vcenter{\hbox{\begin{tikzpicture}[scale=0.3]     
    \node (g) at (-1.5, 3) {};
    \node (3) at (0, 6) {};
    \node (1) at (-3, 6) {};
    \node (c) at (0, 0) {};
    \node (2) at (1.5, 3) {};


    \draw[thick] (g)--(1);
    \draw[thick] (c)--(g);
    \draw[thick] (c)--(2);
    \draw[thick] (g)--(3);


    \draw[circle, fill=black] (c) circle [radius=32pt];
    \draw[circle, fill=black] (g) circle [radius=32pt];
    \draw[circle, fill=white] (1) circle [radius=24pt];
    \draw[circle, fill=white] (2) circle [radius=24pt];
    \draw[circle, fill=white] (3) circle [radius=24pt];
    
    \node at (1) {\scalebox{0.8}{$1$}};
    \node at (2) {\scalebox{0.8}{$2$}};
    \node at (3) {\scalebox{0.8}{$3$}};
    \node at (c) {\scalebox{1}{$\textcolor{white}{e_{(1)}^k}$}};
    \node at (g) {\scalebox{1}{$\textcolor{white}{e_{(2)}^k}$}};

  \end{tikzpicture}}}
  \]
  This shows that the operad $\greg^C$ satisfies Relation \ref{eq:gregDelta}.\\
  Let $\P(x, y, g^1, \dots, g^n)$ be the suboperad of $\greg^C$ generated by $x$, $y$ and $g^1, \dots, g^n$. We have to show that $\P(x, y, g^1, \dots, g^n)=\greg^C$. Let us prove it by induction on the arity.\\
  \emph{Base case:} By definition $\P(x, y, g^1, \dots, g^n)(2)=\greg^C(2)$.\\
  \emph{Induction step:} Let $m\geq 2$ and suppose that $\P(x, y, g^1, \dots, g^n)(k)=\greg^C(k)$ for all $k\leq m$. We have to show that $\P(x, y, g^1, \dots, g^n)(m+1)=\greg^C(m+1)$.\\
  Computing $x\circ_1x_m$ and $x\circ_1g^k_m$ shows that $x_{m+1}\in \P(x, y, g)(m+1)$ and $g^k_{m+1}\in \P(x, y, g)(m+1)$. Since we can obtain any rooted Greg trees by inductively composing corollas in the leaves of smaller trees, we have $\greg^C(m+1)= \P(x, y, g^1, \dots, g^n)(m+1)$.\\
  Hence by induction, $\P(x, y, g^1, \dots, g^n)=\greg^C$.
\end{proof}

\section{Koszul dual and Koszulness}

Let us use the same strategy as in the previous sections to prove that the operad $\greg^C$ is Koszul. 

\begin{Definition}
  Let $\G^C$ the operad defined by generators and relations as follows: $\tx$ is a generator without symmetries, $\tg^k$ are symmetric generators such that:
  \begin{equation*}
    (\tx \circ_1\tx -\tx \circ_2\tx) - (\tx \circ_1\tx -\tx \circ_2\tx).(2\;3)
  \end{equation*}
  \begin{multline*}
    (\tx \circ_1 \tg^k - (\tg^k \circ_1 \tx ).(2\;3) - \tg^k \circ_2 \tx ) - 
    (\tx \circ_1 \tg^k - (\tg^k \circ_1 \tx ).(2\;3) - \tg^k \circ_2 \tx ).(2\;3)\\
    +  (\tg^k_{(1)}\circ_1 \tg^k_{(2)} - (\tg^k_{(1)}\circ_1 \tg^k_{(2)}).(2\;3))  
  \end{multline*}
  With $\tg^k_{(1)}$ and $\tg^k_{(2)}$ defined by $\Delta$ under the identification of $V$ with the span of the generators $\tg^k$.
\end{Definition}  

Let $C^*=(V^*, \mu)$ the linear dual of $C$. This is a commutative algebra of dimension $n$.

\begin{Definition}
  Let $(\G^C)^!$ the operad defined by generators and relations as follows: $x_*$ is a generator without symmetries, $g^k_*$ are skew-symmetric generators such that:
  $$x_*\circ_1x_*-x_*\circ_2x_* \qquad ; \qquad x_*\circ_1x_*-(x_*\circ_1x_*).(2\;3)$$
  $$x_*\circ_1g^k_*-g^k_*\circ_2x_* \qquad ; \qquad x_*\circ_1g^k_*-(g^k_*\circ_1x_*).(2\;3)$$
  $$x_*\circ_1g^k_*+(x_*\circ_1g^k_*).(1\;2\;3)+(x_*\circ_1g^k_*).(1\;3\;2)$$
  $$x_*\circ_2g^k_*  \qquad ; \qquad g^i_*\circ_1g^j_*-x\circ_1g_*^{i.j}$$
  With the notation $g_*^{i.j}=\mu(g_*^i, g_*^j)$, one should be careful since $i.j$ in not the product of $i$ and $j$; this is just a way to keep notation more compact.
\end{Definition}

Let us generalize the construction of $W(\chi)$ of Definition \ref{Definition:walg} to get an example of $(\G^C)^!$-algebra that allows us to show that $\dim((\G^C)^!(4))\geq 4+3n$.

\begin{Definition}
  Let $\chi$ be a finite alphabet and $\mathbf{W}_C(\chi)$ be the span of finite words on $\chi$ with the following extra decorations: either one letter is pointed with a dot or there is an arrow from one letter to another, the arrow is linearly labeled by $V^*$. \\
  Let us write $\overset{i}{\curvearrowright}$ instead of $\overset{e^*_i}{\curvearrowright}$ and $\overset{i.j}{\curvearrowright}$ instead of $\overset{\mu(e^*_i, e^*_j)}{\curvearrowright}$.\\
  Let $W_C(\chi)$ be the quotient of $\mathbf{W}_C(\chi)$ by the following relations, letters commute with each other (the dot and the arrow follow the letter), reverting the arrow changes the sign and 
  $$\overset{\overset{i}{\curvearrowright}}{ab}cv=\overset{\overset{i}{\curvearrowright}}{cb}av+
  \overset{\overset{i}{\curvearrowright}}{ac}bv$$
  for any $a, b, c\in \chi$ and $v$ a finite word. Because the letters commute, we can write the elements of $W(\chi)$ with the pointed letter (or arrowed letters) at the start.\\
  Let the $\ccircled{x}$ and $\ccircled{g}_i$ products on $W(\chi)$ defined by:
  \begin{itemize}
    \item $\dot{a}v\ccircled{x}\dot{b}w=\dot{a}vbw$
    \item $\overset{\overset{i}{\curvearrowright}}{ab}v\ccircled{x}\dot{c}w=
    \overset{\overset{i}{\curvearrowright}}{ab}vcw$
    \item $\dot{a}v\ccircled{g}_i\dot{b}w=\overset{\overset{i}{\curvearrowright}}{ab}vw$
    \item $\overset{\overset{i}{\curvearrowright}}{ab}v\ccircled{g}_j\dot{c}w=
    \overset{\overset{i.j}{\curvearrowright}}{ab}vcw$
  \end{itemize}
  All other cases give $0$.
\end{Definition} 

\begin{Proposition}
  The algebra $(W_C(\chi), \ccircled{x}, \{ \ccircled{g}_i \} )$ is a $(\G^C)^!$-algebra generated by $\chi$.
\end{Proposition}

\begin{proof}
  Indeed, $\dot{a}v=\dot{a}\ccircled{x}w$ with $w$ the word $v$ with a dot on a letter (let us say the first one for example) and 
  $$\overset{\overset{i}{\curvearrowright}}{ab}v=(\dot{a}\ccircled{g}_i\dot{b})\ccircled{x}w$$ 
  so $(W(\chi), \ccircled{x}, \{ \ccircled{g}_i\} )$ it is generated by $\chi$. The products $\ccircled{g}_i$ are skew-symmetric since $$\dot{a}\ccircled{g}_i\dot{b}=\overset{\overset{i}{\curvearrowright}}{ab}= \overset{\overset{i}{\curvearrowleft}}{ba}=-\overset{\overset{i}{\curvearrowright}}{ba}=-\dot{b}\ccircled{g}_i\dot{a}$$
  The $7$ relations of $(\G^C)^!$ are easily checked.
\end{proof}

\begin{Proposition}
  We have $\dim((\G^C)^!(4))\geq 4+3n$.
\end{Proposition}

\begin{proof}
  Let $A=W(\{a, b, c, d\})$. Let us compute the dimension of $\Mult(A(4))$ the multilinear part of $A(4)$. Let us write words of $\Mult(A(4))$ such that the arrow always starts from $a$ and go to the second letter, and aside from that, the letters are in the lexicographic order.\\
  The rewriting rule $\overset{\overset{i}{\curvearrowright}}{\beta\gamma}\alpha\mapsto \overset{\overset{i}{\curvearrowright}}{\alpha\gamma}\beta- \overset{\overset{i}{\curvearrowright}}{\alpha\beta}\gamma$ is confluent. Indeed: 
  $$\overset{\overset{i}{\curvearrowright}}{cd}ab=
  \overset{\overset{i}{\curvearrowright}}{bd}ac-\overset{\overset{i}{\curvearrowright}}{bc}ad= (\overset{\overset{i}{\curvearrowright}}{ad}bc-\overset{\overset{i}{\curvearrowright}}{ab}cd)- (\overset{\overset{i}{\curvearrowright}}{ac}bd-\overset{\overset{i}{\curvearrowright}}{ab}cd)= \overset{\overset{i}{\curvearrowright}}{ad}bc-\overset{\overset{i}{\curvearrowright}}{ac}bd$$
  Hence, $\Mult(A(4))$ is spanned by the words $\dot{a}bcd$, $a\dot{b}cd$, $ab\dot{c}d$, $abc\dot{d}$, $\overset{\overset{i}{\curvearrowright}}{ab}cd$, $\overset{\overset{i}{\curvearrowright}}{ac}bd$ and $\overset{\overset{i}{\curvearrowright}}{ad}bc$. Since $A$ is a $(\G^C)^!$-algebra on $4$ generators, \hbox{$\dim((\G^C)^!(4))\geq\dim(\Mult(A(4)))=4+3n$}.
\end{proof}

\begin{Remark}
  The algebra $(W(\chi), \ccircled{x}, \{ \ccircled{g}_i\} )$ is in fact the free $(\G^C)^!$-algebra generated by $\chi$. 
\end{Remark}

Let us consider the three following orders to get Gröbner bases:
\begin{itemize}
  \item the degree-lexicographic permutation order with $x_*>y_*>g_*$ gives Figure \ref{fg:StGBg1};
  \item the weighted permutation reverse-degree-lexicographic order with $g_*>y_*>x_*$ and $g_*$
  of degree $1$ gives Figure \ref{fg:StGBg2};
  \item and the reverse-degree-lexicographic permutation order with $x_*>y_*>g_*$ gives Figure \ref{fg:StGBg3}.
\end{itemize}

\begin{Proposition}
  The rewriting systems displayed in Figure \ref{fg:StGBg1}, Figure \ref{fg:StGBg2} and Figure \ref{fg:StGBg3} are confluent. They are in fact Gröbner bases.
\end{Proposition}

\begin{proof}
  As in Proposition \ref{Proposition:GB}, remarking that there are $4+3m$ normal forms in arity $4$ for each rewriting system and that $\dim((\G^C)^!(4))\geq 4+3n$ is enough. Since we have a monomial order and the rewriting rules are quadratic in a binary operad, checking arity $4$ is enough.
\end{proof}

\begin{Proposition}
  We have that $\dim((\G^C)^!(m))=(n+1)m-n$ for all $m\geq 1$, hence its exponential generating series is $((n+1)t-n)\exp(t)+n$.
\end{Proposition}

\begin{proof}
  It suffices to count the normal forms of a rewriting system for example Figure \ref{fg:StGBg1}. Let $m\geq 2$ and count the number of normal forms in arity $m$. Those are right combs with at most one $g^k_*$, with all the $x_*$ above the $g^k_*$ and the $y_*$, and all the $y_*$ below the $g^k_*$ and the $x_*$. Hence, the normal forms are determined by the number of occurrences $x_*$ and $g^k_*$, and have either zero or one occurrence $g^k_*$. If there is no $g^k_*$, then one can have from $0$ to $m-1$ occurrences \hbox{of $x_*$.} If there is one $g^k_*$, then one can have from $0$ to $m-2$ occurrences of $x_*$ and $n$ choice for the $g^k_*$ that appears. Hence, the number of normal forms in arity $m$ is $m+n(m-1)=(n+1)m-n$.
\end{proof}

\begin{Theorem}
  The operad $\G^C$ is Koszul.
\end{Theorem}

\begin{proof}
  We have a quadratic Gröbner basis for $(\G^C)^!$, hence $(\G^C)^!$ is Koszul, hence $\G^C$ is Koszul.
\end{proof}

\begin{Proposition}\label{Proposition:egsgc}
  The exponential generating series of $\greg^C$ verifies:
  $$ f_{\greg^C}=t\exp(f_{\greg^C})+n(\exp(f_{\greg^C})-f_{\greg^C}-1)$$
\end{Proposition}

\begin{proof}
  An inspection of the species $\G^V$ which is the species of rooted Greg trees such that the black vertices are labeled by $\{e_1, \dots, e_n\}$ shows that: 
  $$ \G^V=X\cdot E(\G^V) + nE_{\geq2}(\G^V)$$
  With the usual notation of species, $X$ is the singleton species, $E$ is the species of sets and $E_{\geq2}$ is the species of sets with at least two elements.\\
  The above equation means that a rooted Greg tree is either a white vertex and a set of rooted Greg trees connected to it, or a black vertex labeled by $e_k$ (so $n$ possibilities) and a set of at least $2$ rooted Greg trees connected to it.\\
  Since $\G^V$ is the underlying species of $\greg^C$, we have that:
  $$ f_{\greg^C}=t\exp(f_{\greg^C})+n(\exp(f_{\greg^C})-f_{\greg^C}-1)$$
\end{proof}

\begin{Remark}
  We can recover the recursive formula enumerating the rooted Greg trees from \cite[Proposition 2.1]{GenGreg} by resolving a differential equation:\\
  We have $h=((n+1)t+n)\exp(-t)-n$. Hence: 
  \begin{itemize}
    \item $d_1h=-((n+1)t-1)\exp(-t)$, 
    \item $d_2h=(t+1)\exp(-t)-1$.
  \end{itemize} 
  Hence:
  $$(z+2)h + d_1h -(z+1)^2d_2h=1 $$
  Let $f$ be such that $h(f(t, n), n)=h\circ(f, \id)=t$. We have that $d_1h\circ(f, \id).d_1f=1$ and\\
  $d_1h\circ(f, \id).d_2f+d_2h\circ(f, \id)=0$, hence:
  $$((z+2)t-1)d_1f+(z+1)^2d_2f=-1$$
  Let $f(t, n)=\sum\frac{g_k(n)}{k!}t^k$ with $g_k$ polynomials in $n$, we get the following recursive relation:
  \begin{itemize}
    \item $g_1(n)=1$
    \item $g_{k+1}(n)=(n+2)kg_k(n)+(n+1)^2g_k'(n)$\\
  \end{itemize}
\end{Remark}

\begin{Theorem}
  The operad $\G^C$ is isomorphic to $\greg^C$.
\end{Theorem}

\begin{proof}
  We know that $\dim((\G^C)^!(m))=(n+1)m-n$, hence its exponential generating series is $f_{(\G^C)^!}=((n+1)t-n)\exp(t)+n$. Let $h(t, n)=-f_{(\G^C)^!}(-t)=((n+1)t+n)\exp(-t)-n$, since $\G^C$ is Koszul, we know that $h(f_{\G^C}(t, n), n)=t$. Hence:
  $$ t=((n+1)f_{\G^C}+n)\exp(-f_{\G^C})-n$$
  Hence:
  $$ f_{\G^C}=t\exp(f_{\G^C})+n(\exp(f_{\G^C})-f_{\G^C}-1)$$
  Which shows that $f_{\G^C}=f_{\greg^C}$. Since we have a surjective morphism from $\G^C$ to $\greg^C$ and equality of dimensions of components, we have that $\G^C$ is isomorphic to $\greg^C$.
\end{proof} 

\begin{Corollary}\label{Theorem:m2}
  The operad $\greg^C$ is binary, quadratic and Koszul.
\end{Corollary}

\begin{Definition}
  Let $\greg_n$ be the operad $\greg^{(V, 0)}$, with $0$ the trivial coproduct on $V$.
\end{Definition}

\begin{Corollary}
  The operad $\bigvee_\Lie^{n+1}\PL$ is isomorphic to $\greg^{(V, \Delta_{\max})}$ with:
  $$\Delta_{\max}:e_k\mapsto \sum_{i, j\vert \max(i, j)=k}e_i\otimes e_j$$
  Moreover, $\bigvee_\Lie^{n+1}\PL$ is filtered by the grading of the rooted Greg trees by the number of black vertices. The associated graded operad is $\greg_n$.
\end{Corollary}

\begin{Corollary}
  The operad $\bigvee_\Lie^{n+1}\PL$ is Koszul.
\end{Corollary}

\begin{Remark}\label{Remark:PT}
  This fact is not a direct consequence of the definition of $\bigvee_\Lie^{m+1}\PL$ as a coproduct. Indeed, the fiber coproduct of two Koszul operads $\P$ and $\Q$ over a Koszul operad $\R$ is not necessarily Koszul. Take for instance the operads $\bigvee_\Lie^2\Ass$ and $\bigvee_\Lie^2\Pois$ which are not Koszul, it can be checked by comparing the exponential generating series of those operads and of their Koszul dual. Worst, freeness as right $\R$-modules of $\P$ and $\Q$ does not solve the issue as shown by the example $\bigvee_\Lie^2\Pois$. It seems that freeness as left modules solve this issue. Indeed, for instance the operads $\bigvee_{\Com}^2\Pois$ and $\bigvee_{\Com}^2\Zin$ are Koszul. However, the author does not know how to prove that left freeness ensure that Koszulness is preserved. Left and right freeness are defined at the very beginning of the next section.
\end{Remark}

\section{Nielsen-Schreier and freeness properties}\label{seq:free}

We have seen that $\greg^C$ is Koszul using quadratic Gröbner bases. However, one Gröbner basis was enough to show this fact. Three different Gröbner bases were computed with particular normal forms. Indeed, theorems from Dotsenko \cite{Free} and Dotsenko and Umirbaev \cite{NSPrp} allow us to show some freeness properties using Gröbner bases. Let us recall those freeness properties and show that they hold for $\greg^C$.
\paragraph{}
Let $\P$ an operad. 
\begin{itemize}
  \item A \emph{left module} $L$ over $\P$ is a species $L$ with a morphism
$\P\circ L\to L$ satisfying the usual axioms.
  \item A \emph{right module} $R$ over $\P$ is a species $R$ with a morphism
  $R\circ \P\to R$ satisfying the usual axioms.
  \item A \emph{bimodule} over $\P$ is a left and right module over $\P$ such that
the two structures commute.
\end{itemize}
Let $\Q$ be an operad such that we have a morphism of operads $\P\to\Q$, then $\Q$ has a canonical structure of a left and a right module over $\P$. (It has in fact a structure of bimodule over $\P$.) A left module $L$ (resp. right module $R$) over $\P$ is said to be \emph{free} if $L$ (resp. $R$) is isomorphic to $\P\circ \X$ (resp. $\X\circ \P$) for some species $\X$ and that the module structure is given by the operadic composition in $\P$. In this context, the $\circ$ is the plethysm of species, which can be interpreted as the composition of Schur functors. We refer to \cite{Species} for more details on the plethysm of species.
\paragraph{}
Let $\F(E)/(R)$ and $\F(E\sqcup F)/(R\sqcup S)$ be presentations of the operads $\P$ and $\Q$ respectively, such that $R\sqcup S$ is a Gröbner basis for some monomial order and $R$ is a Gröbner basis once the monomial order is restricted to shuffle trees over $E$. Let recall the following theorems:

\begin{Theorem}[\emph{left freeness version}] \cite[Theorem 4]{Free}
  Assume that the root of the leading terms of $S$ are elements of $F$. Then $\Q$ is free as left $\P$-module.
\end{Theorem}

\begin{Theorem}[\emph{right freeness version}] \cite[Theorem 4]{Free}
  Assume that the vertices such that each child is a leaf, of the leading terms of $S$ are elements of $F$. Then $\Q$ is free as right $\P$-module.
\end{Theorem}

Let $C$ be a coassociative cocommutative coalgebra and $C'$ a subcoalgebra of $C$.

\begin{Corollary}\label{Theorem:m3}
  The operad $\greg^C$ is free as left and as a right $\greg^{C'}$-module (and not as a bimodule).
\end{Corollary}

\begin{proof}
  By reversing the order, one can go from a Gröbner basis of an operad to a Gröbner basis of its Koszul dual, this exchanges the leading terms and the normal forms. Hence, Gröbner basis from Figure \ref{fg:StGBg2} witness the left freeness and Gröbner basis from Figure \ref{fg:StGBg3} witness the right freeness.
\end{proof}

\begin{Corollary}\label{Theorem:m4}
  The operad $\bigvee_\Lie^{n+1}\PL$ is free as left and as a right $\bigvee_\Lie^n\PL$-module (and not as a bimodule).
\end{Corollary}

Let us now define the Nielsen-Schreier property.

\begin{Definition}
  An operad $\P$ has the \emph{Nielsen-Schreier property} if the subalgebra of any free $\P$-algebra is free.
\end{Definition}

Let us recall the following theorem:

\begin{Theorem}\cite[Theorem 4.1]{NSPrp}
  Let $\P$ an operad and $E$ a set of generator of $\P$ satisfying the following
  conditions:
  \begin{itemize}
    \item $P$ admits a Gröbner basis for the reverse path lexicographic ordering such that for each leading term, the smallest leaf is directly connected to the root.
    \item $P$ admits a Gröbner basis such that each leading term are left combs such that the smallest leaf and the second-smallest leaf are directly connected to the same vertex.
  \end{itemize}
  Then $\P$ has the Nielsen-Schreier property.
\end{Theorem}

\begin{Corollary}\label{Theorem:m3ns}
  The operad $\greg^C$ has the Nielsen-Schreier property.
\end{Corollary}

\begin{proof}
  The Gröbner basis from Figure \ref{fg:StGBg1} witness the first condition and the Gröbner basis from Figure \ref{fg:StGBg2} witness the second condition.
\end{proof}

\begin{Corollary}\label{Theorem:m4ns}
  The operad $\bigvee_\Lie^n\PL$ has the Nielsen-Schreier property.
\end{Corollary}

\section{Explicit computation of the generators}

Let us compute the explicit generators of $\bigvee_\Lie^{n+1}\PL$ as a left $\bigvee_\Lie^n\PL$-module. To do so, let us mimic the proof of Dotsenko \cite{FPRed}.
\paragraph{}
A structure of cyclic operad on an operad $\P$ is given by an action of $\mathfrak{S}_{n+1}$ on $\P(n)$ compatible with the operadic structure. This is equivalently given by an action of $\tau=(1, 2, \ldots, n+1)$ on $\P(n)$ verifying:
\begin{itemize}
  \item $\tau(\mu\circ_i\nu)=\tau(\mu)\circ_{i+1}\nu$ for $i<m$ with $m$ the arity of $\mu\circ_i\nu$;
  \item $\tau(\mu\circ_m\nu)=\tau(\nu)\circ_1\tau(\mu)$.
\end{itemize}
It is known that $\Lie$ is a cyclic operad, $\CycLie$ is the species under this cyclic operad so as vector space we have $\Lie(k)=\CycLie(k+1)$. In the particular case of $\CycLie$, the action of $\tau$ is given by $\tau(l)=l$. Let us introduce the following notation, $x$ is the generator of $\PL$ without symmetries, $x=\mu+l$ with $\mu$ symmetric and $l$ skew-symmetric. Then $l$ is the generator of the suboperad $\Lie$ of $\PL$ and $\mu$ is magmatic.
\paragraph{}
We want to prove that $\bigvee_\Lie^{n+1}\PL\simeq\bigvee_\Lie^n\PL\circ\F(\bar{\F}^{(n)}(\CycLie))$ with $\F$ the free operad functor, $\bar{\F}$ the reduced free operad functor such that $\F(\X)=\bar{\F}(\X)\oplus\X$ and $\bar{\F}^{(n)}$ the $n$-th iteration of $\bar{\F}$.
\paragraph{}
Let us recall the following theorem from \cite{FPRed}:

\begin{Theorem}
  Let $\Y$ be the subspecies of $\PL$ such that $y\in\Y$ if and only if $y=(\mu\circ_2 a)\circ_1 b$ with $a, b\in\Lie$. Then:
  \begin{itemize}
    \item $\Y$ is isomorphic to $\CycLie$ as species;
    \item Let $\P(\Y)$ the suboperad of $\PL$ generated by $\Y$, then $\P(\Y)$ is free;
    \item The left $\Lie$-submodule of $\PL$ generated by $\P(\Y)$ is free and coincide with $\PL$.
  \end{itemize}
\end{Theorem}

Let us use the exact same technic as the one used in \cite{FPRed} to explicitly compute the generators of $\bigvee_\Lie^{n+1}\PL$ as a left $\bigvee_\Lie^n\PL$-module. The idea is the following, we introduce some explicit generators of $\bigvee_\Lie^{n+1}\PL$ as a left $\Lie$-module and as a left $\bigvee_\Lie^n\PL$-module. We then define a bunch of surjective morphisms of species involving those generators. We then compute the dimensions of the species involved to show that the morphisms are isomorphisms. We then conclude that the generators we introduced freely generate $\bigvee_\Lie^{n+1}\PL$. Let $\Y_n$ the subspecies of $\bigvee_\Lie^{n+1}\PL$ generated by the $(c_n\circ_2 a)\circ_1 b$ such that $a, b\in\bigvee_\Lie^n\PL$. Let $\X_n$ the subspecies of $\Y_n$ generated by the $x=(c_n\circ_2 a)\circ_1 b$ with $a, b\in\Lie$. Let $\P(\Y_n)$ the suboperad of $\bigvee_\Lie^{n+1}\PL$ generated by $\Y_n$. And let $\Z_n$ the species inductively defined by $\Z_0=\F(\Y_0)$ and $\Z_{n+1}=\Z_n\circ\F(\Y_{n+1})$.

\begin{Lemma}
  We have a surjective morphism of $\Lie$-bimodule from $\Lie\circ\Z_n$ to $\bigvee_\Lie^{n+1}\PL$.
\end{Lemma}

\begin{proof}
  Since $\Y_k$ is a subspecies of $\bigvee_\Lie^{n+1}\PL$, we have a morphism of species from $\F(\Y_k)$ to $\bigvee_\Lie^{n+1}\PL$. Hence, we have a morphism from $\Z_n$ to $\bigvee_\Lie^{n+1}\PL$. Since $\Lie$ is a suboperad of $\bigvee_\Lie^{n+1}\PL$, we have a morphism of left $\Lie$-module from $\Lie\circ\F(\Y_k)$ to $\bigvee_\Lie^{n+1}\PL$. This is a morphism of right $\Lie$-module since each $\F(\Y_k)$ is a right $\Lie$-module. Moreover, this morphism is surjective since $l, \mu, c_1, \dots, c_n$ are in $\Lie\circ\Z_n$ and are generators of $\bigvee_\Lie^{n+1}\PL$.
\end{proof}

Let us define the following filtration on $\Lie\circ\Z_n$:

\begin{Definition}
  Let us define the weight of an element of $\F(\Y_k)$ as the usual weight in free operad, which is the number of generators needed in the composition. Then we define inductively the weight of an element $\gamma(z, f_1, \dots, f_k)$ of $\Z_n$ with $z\in\Z_{n-1}$ and $f_i\in\F(\Y_n)$ as the total sum of the weight of those elements. For an element $\alpha=\gamma(l, z_1, \dots, z_r) $ of $\Lie\circ\Z_n$ such that $z_i\in\Z_n$ of weight $w_i$ and $l\in\Lie$ of arity $r$, let $w=r+\sum w_i$ be the weight of $\alpha$. We define the filtration by $\alpha\in F^w(\Lie\circ\Z_n)$ with $w$ the weight of $\alpha$.
\end{Definition}

\begin{Proposition}
  This filtration is compatible with the $\Lie$-bimodule structure. (It is in fact a filtration by infinitesimal $\Lie$-bimodule.)
\end{Proposition}

\begin{proof}
  Indeed, we have that $l(F^p(\Lie\circ\Z_n), F^q(\Lie\circ\Z_n))\subseteq F^{p+q}(\Lie\circ\Z_n)$ and\\
  $F^p(\Lie\circ\Z_n)\circ\Lie\subseteq F^p(\Lie\circ\Z_n)$.
\end{proof}

Hence, this filtration induces a filtration on $\bigvee_\Lie^{n+1}\PL$ by the surjective morphism of $\Lie$-bimodule of the previous lemma.

\begin{Lemma}
  We have a surjective morphism of species from $\CycLie$ to $\X_n$.
\end{Lemma}

\begin{proof}
  Let us compute Relation (\ref{eq:vpl}) with $x=\mu+l$. We get:
  \begin{multline*}
    (l \circ_1 c_n - (c_n \circ_1 l ).(2\;3) - c_n \circ_2 l ) - 
  (l \circ_1 c_n - (c_n \circ_1 l ).(2\;3) - c_n \circ_2 l ).(2\;3)\\
  (\mu \circ_1 c_n - (c_n \circ_1 \mu ).(2\;3) - c_n \circ_2 \mu ) - 
  (\mu \circ_1 c_n - (c_n \circ_1 \mu ).(2\;3) - c_n \circ_2 \mu ).(2\;3)\\
   +  \sum_{i, j\mid \max(i, j)=n}(c_i\circ_1 c_j - (c_i\circ_1 c_j).(2\;3))
  \end{multline*}
  Let us rewrite it a bit:
  \begin{multline*}
    2\times (c_n \circ_2 l) + (c_n \circ_1 l ).(2\;3) - (c_n \circ_1 l )=\\
    l \circ_1 c_n - (l \circ_1 c_n ).(2\;3)+\\    
  (\mu \circ_1 c_n - (c_n \circ_1 \mu ).(2\;3)) - 
  (\mu \circ_1 c_n - (c_n \circ_1 \mu ).(2\;3)).(2\;3)\\
   +  \sum_{i, j\mid \max(i, j)=n}(c_i\circ_1 c_j - (c_i\circ_1 c_j).(2\;3))
  \end{multline*}
  Let us remark that element of the left-hand side are in $F^2\bigvee_\Lie^{n+1}\PL$ and elements of the right-hand side are in $F^3\bigvee_\Lie^{n+1}\PL$. Indeed, at the left-hand side, we have composition of the identity (arity $1$) with elements of $\F(\Y_n)$ having exactly one occurrence of an element of $\{\mu, c_1, \dots, c_n\}$, hence degree $2$. At the right-hand side, we have either composition of  $l$ (arity $2$) with elements of $\F(\Y_n)$ having exactly one occurrence of an element of $\{\mu, c_1, \dots, c_n\}$, hence degree $3$; or composition the identity (arity $1$) with elements of $\F(\Y_n)$ having exactly two occurrences of an element of $\{\mu, c_1, \dots, c_n\}$, hence degree $3$.\\
  Let us consider $\gr_F\X_n$ the graded species associated to the restriction of the filtration $F$ of $\X_n$. As species we have that $\gr_F\X_n$ is isomorphic to $\X_n$. Moreover, in $\gr_F\X_n$, the above relation gives:
  \begin{equation*}
    2\times (c_n \circ_2 l) + (c_n \circ_1 l ).(2\;3) - (c_n \circ_1 l )=0
  \end{equation*}
  Let us denote $r$ this relation and compute $\frac{1}{3}(r+r.(1\;3))$:
  \begin{equation*}
    \frac{1}{3}(2\times (c_n \circ_2 l) + (c_n \circ_1 l ).(2\;3) - (c_n \circ_1 l )+
    2\times (c_n \circ_2 l).(1\;3) + (c_n \circ_1 l ).(1\;2\;3) - (c_n \circ_1 l ).(1\;3))=0
  \end{equation*}
  We get:
  \begin{equation}\tag{cyc}
    (c_n \circ_2 l) = (c_n \circ_1 l )
  \end{equation}
  This relation allows us to define a morphism of species from $\CycLie$ to $\gr_F\X_n$ by sending $\tilde{\id}\mapsto c_n$ and $\tilde{l}\mapsto (c_n \circ_2 l)$, with $\tilde{\id}$ and $\tilde{l}$ the identity and the Lie bracket of $\CycLie$. Indeed, we have an action of $\tau=(1\;2\;3)$ on $c_n \circ_2 l$ and $(c_n \circ_2 l).\tau=(c_n \circ_1 l)$, which gives $(c_n \circ_2 l)$ by the relation above, hence $\tau(\tilde{l})=\tilde{l}$. This morphism is surjective since $\gr_F\X_n$ is a right $\Lie$-module generated by $c_n$. Hence, we have a surjective morphism of species from $\CycLie$ to $\X_n$.
\end{proof}

\begin{Lemma}
  We have a surjective morphism of species from $\X_n\circ\Z_{n-1}$ to $Y_n$.
\end{Lemma}

\begin{proof}
  Let $\gamma(c_n, y_1, \dots, y_k)$ a monomial element of $\Y_n$, since $y_i\in\bigvee_\Lie^n\PL$ and 
  $\Lie\circ\Z_{n-1}$ surjects on $\bigvee_\Lie^n\PL$, we have $l_i$ such that:
  $$y_i=\gamma(l_i, \alpha_{(i, 1)}, \dots, \alpha_{(i, r_i)})$$
  with $l_i\in\Lie$ and $\alpha_{(i, j)}\in\Z_{n-1}$. Let $\beta=\gamma(c_n, l_1, \dots, l_k)$, we have $\beta\in\X_n$, hence $\gamma(c_n, y_1, \dots, y_k)$ is in the image of $\X_n\circ\Z_{n-1}$.
\end{proof}

Let us summarize the morphisms of species we have:
\begin{multline*}\label{eq:surj}
\bigvee_\Lie^n\PL\circ\F(\CycLie\circ\Z_{n-1})\twoheadrightarrow \bigvee_\Lie^n\PL\circ\F(\X_n\circ\Z_{n-1})
\twoheadrightarrow\\
\bigvee_\Lie^n\PL\circ\F(\Y_n)\twoheadrightarrow \bigvee_\Lie^{n+1}\PL
\end{multline*}
One last ingredient is needed: the equality of dimensions of the components to show that those morphisms are in fact isomorphisms.

\begin{Proposition}
  Let $S$ a species, $f_S(t)$ its exponential generating series. Then 
  $$f_{\F(\bar{\F}^{(n)}(S))}(t)=\frac{\rev_t(t-(n+1)f_S(t))-t}{n+1}+t$$
  where $\rev_t$ is the inverse of the composition in the argument $t$ and $f_{\F(\bar{\F}^{(n)}(S))}$ the exponential generating series of $\F(\bar{\F}^n(S))$.
\end{Proposition}

\begin{proof}
  For a species $S$ with exponential generating series $f_S(t)$, the exponential generating series $f_{\F(S)}(t, z)$ of $\F(S)$ is given by $f_{\F(S)}(t, z)$ which is the inverse of $t-zf_S(t)$ for the composition in the argument $t$, hence we have $f_{\F(S)}(t, z)=\rev_t(t-zf_S(t))$. Hence, the exponential generating series of $\bar{\F}(S)$ is $f_{\bar{\F}(S)}(t, z)=f_{\F(S)}(t, z)-t=\rev_t(t-zf_S(t))-t$. The exponential generating series $f_{\F(S)}(t, z)$ and $f_{\bar{\F}(S)}(t, z)$ have two arguments, the first one $t$ count the arity of the elements and the second one $z$ count the number of generators of the elements in the free operad. Since $z$ count the number of generators of the elements in the free operad, dividing by $z$ allows us to count the number of compositions of generators. Hence, the exponential generating series of $\bar{\F}^{(n)}(S)(t)$ is $\frac{\rev_t(t-nf_S(t))-t}{n}$. Finally, we get:
  $$f_{\F(\bar{\F}^{(n)}(S))}(t)=\frac{\rev_t(t-(n+1)f_S(t))-t}{n+1}+t$$
\end{proof}

\begin{Lemma}
  The exponential generating series of $\bigvee_\Lie^n\PL\circ\F(\bar{\F}^{(n)}(\CycLie))$ is equal to the exponential generating series of $\bigvee_\Lie^{n+1}\PL$.
\end{Lemma}

\begin{proof}
  Let us compute the exponential generating series of $\F(\bar{\F}^{(n)}(\CycLie))$. The exponential generating series of $\CycLie$ is well known to be $(1-t)\ln(1-t)+t$, indeed its dimensions are $(n-2)!$. Hence, the exponential generating series of $\F(\bar{\F}^{(n)}(\CycLie))$ is 
  $$f_{\F(\bar{\F}^{(n)}(\CycLie))}(t)=\frac{\rev_t(t-(n+1)(1-t)\ln(1-t)-(n+1)t)-t}{n+1}+t$$
  We have already computed the exponential generating series of $\bigvee_\Lie^{n+1}\PL$ in Proposition \ref{Proposition:egsgc} which is
  $$f_{\bigvee_\Lie^{n+1}\PL}(t)=\rev_t((nt+t+n)\exp(-t)-n)$$
  And the exponential generating series of $\bigvee_\Lie^{n}\PL$ is 
  $$f_{\bigvee_\Lie^n\PL}(t)=\rev_t((nt+n-1)\exp(-t)-n+1)$$
  Let us show that $f_{\bigvee_\Lie^{n+1}\PL}(t)=(f_{\bigvee_\Lie^n\PL}\circ f_{\F(\bar{\F}^{(n)}(\CycLie))})(t)$. Let 
  $$ f(t)=(nt+t+n)\exp(-t)-n \qquad g(t)=(nt+n-1)\exp(-t)-n+1 $$
  $$ h(t)=t-(n+1)(1-t)\ln(1-t)-(n+1)t$$
  We want to show that: 
  $$\rev_t(f)(t)=(\rev_t(g)\circ(\frac{\rev_t(h)-t}{n+1}+t))(t)$$
  It suffices to show that $h((n+1)g-nf)=f$. Let us compute:
  \begin{align*}
      (n+1)g(t)-nf(t)\hspace{-10pt}&&=&(n+1)((nt+n-1)\exp(-t)-n+1)-n((nt+t+n)\exp(-t)-n) \\
      &&=&((n+1)nt\exp(-t)+(n+1)(n-1)\exp(-t)-(n+1)(n-1))- \\
      && &(n(n+1)t\exp(-t)-n^2\exp(-t)+n^2)   \\
      &&=&-\exp(-t)+1 
  \end{align*}
  Hence:
  $$h((n+1)g-nf)=-\exp(-t)+1 -(n+1)\exp(-t)\ln(\exp(-t))-(n+1)(-\exp(-t)+1 )=f$$ 
  which concludes the proof.
\end{proof}

We can state and prove the generalization of the previous theorem:

\begin{Theorem}\label{Theorem:m5}
  We have:
  \begin{enumerate}
    \item The species $\X_n$ is isomorphic to $\CycLie$ as a species.
    \item The species $\Y_n$ is isomorphic to $\bar{\F}^{(n)}(\CycLie)$ as species;
    \item The suboperad $\P(\Y_n)$ of $\bigvee_\Lie^{n+1}\PL$ generated by $\Y_n$ is free;
    \item The left $\bigvee_\Lie^n\PL$-submodule of $\bigvee_\Lie^{n+1}\PL$ generated by $\P(\Y_n)$ is free and coincide
    with $\bigvee_\Lie^{n+1}\PL$.
    \item The species $\Z_n$ is isomorphic to $\F(\CycLie)\circ\dots\circ\F(\bar{\F}^{(n)}(\CycLie))$ as species;
    \item The left $\Lie$-submodule of $\bigvee_\Lie^{n+1}\PL$ generated by $\Z_n$ is free and coincide with the $\Lie$-module $\bigvee_\Lie^{n+1}\PL$.
  \end{enumerate}
\end{Theorem}

\begin{proof}
  Let us prove this theorem by induction on $n$. The base case is the theorem from \cite{FPRed}.\\
  From the previous lemmas we have 
  \begin{multline*}
    \bigvee_\Lie^n\PL\circ\F(\CycLie\circ\Z_{n-1})\twoheadrightarrow \bigvee_\Lie^n\PL\circ\F(\X_n\circ\Z_{n-1})
    \twoheadrightarrow\\
    \bigvee_\Lie^n\PL\circ\F(\Y_n)\twoheadrightarrow\bigvee_\Lie^n\PL\circ\P(\Y_n)\twoheadrightarrow \bigvee_\Lie^{n+1}\PL
  \end{multline*}
  By item (5), we have $\Z_{n-1}\simeq\F(\CycLie)\circ\dots\circ\F(\bar{\F}^{(n-1)}(\CycLie))$, hence  
  $$ \CycLie\circ\Z_{n-1}\simeq\CycLie\circ\F(\CycLie)\circ\dots\circ\F(\bar{\F}^{(n-1)}(\CycLie)) \simeq \bar{\F}^{(n)}(\CycLie)$$
  Those surjective morphisms are isomorphisms by equality of dimensions. This shows that:
  \begin{enumerate}
    \item The species $\X_n$ is isomorphic to $\CycLie$;
    \item The species $\Y_n$ is isomorphic to $\bar{\F}^{(n)}(\CycLie)$;
    \item The species $\P(\Y_n)$ is isomorphic to $\F(\Y_n)$;
    \item And the left $\bigvee_\Lie^n\PL$-module $\bigvee_\Lie^n\PL\circ\P(\Y_n)$ is isomorphic to $\bigvee_\Lie^{n+1}\PL$ as left
    $\bigvee_\Lie^n\PL$-module.
  \end{enumerate}
  Moreover, since $\Z_n=\Z_{n-1}\circ\F(\Y_n)$ we have that $\Z_n$ is isomorphic to:
  $$\F(\CycLie)\circ\dots\circ\F(\bar{\F}^{(n)}(\CycLie))$$ 
  as a species. Since $\bigvee_\Lie^n\PL\circ\P(\Y_n)$ is isomorphic to $\bigvee_\Lie^{n+1}\PL$ as left $\bigvee_\Lie^n\PL$-module, they are isomorphic as left $\Lie$-module. Hence, $\Lie\circ\Z_n$ is isomorphic to $\bigvee_\Lie^{n+1}\PL$ as left $\Lie$-module, in particular the left $\Lie$-submodule generated by $\Z_n$ is free and coincide with $\bigvee_\Lie^{n+1}\PL$.
\end{proof}

\begin{Remark}
  This proof can be adapted to show that $\greg_n\simeq\greg_{n-1}\circ\F(\bar{\F}^{(n)}(\CycLie))$.
\end{Remark}

The operad $\bigvee_\Lie^{n+1}\PL$ is also free as a right $\bigvee_\Lie^n\PL$-module. It could be interesting to compute explicit generator in this case.

\section*{\hypertarget{Appendix}{Appendix}}

\afterpage{\clearpage}
\begin{center}
\begin{figure}[H]\caption{The 25 relations of the operad $(\greg')^!$}\label{fg:StRel1}
  \begin{center}
  \includegraphics{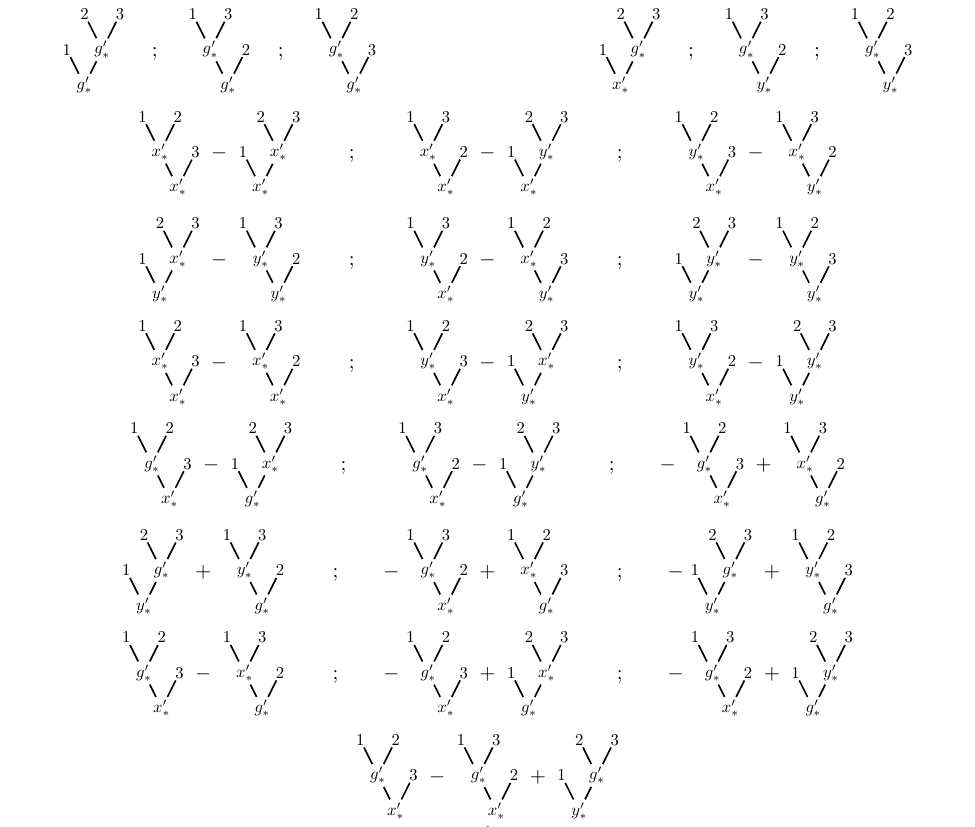}
  \end{center}
\end{figure}
\end{center}

\afterpage{\clearpage}
\begin{center}
\begin{figure}[H]\caption{The ``dlp'' rewriting system for $(\greg')^!$}\label{fg:StGB1}
  \begin{center}
  \includegraphics{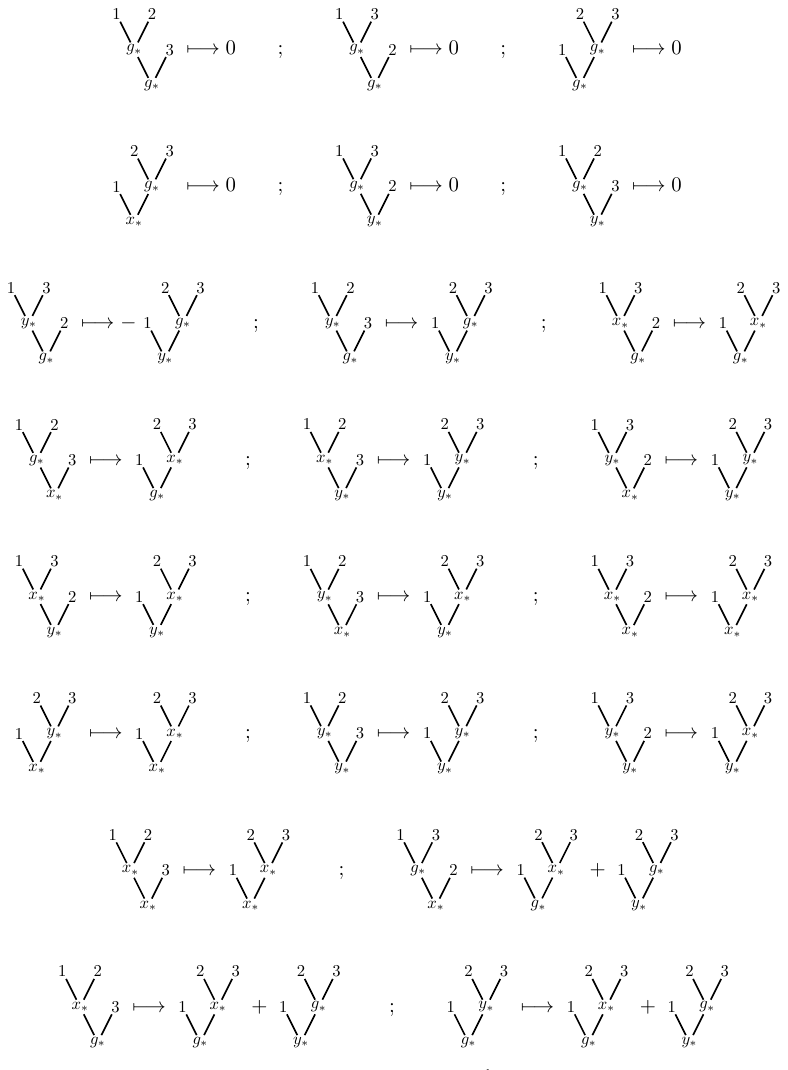}
  \end{center}
\end{figure}
\end{center}

\afterpage{\clearpage}
\begin{center}
\begin{figure}[H]\caption{The ``prdl'' rewriting system for $(\greg')^!$}\label{fg:StGB2}
  \begin{center}
  \includegraphics{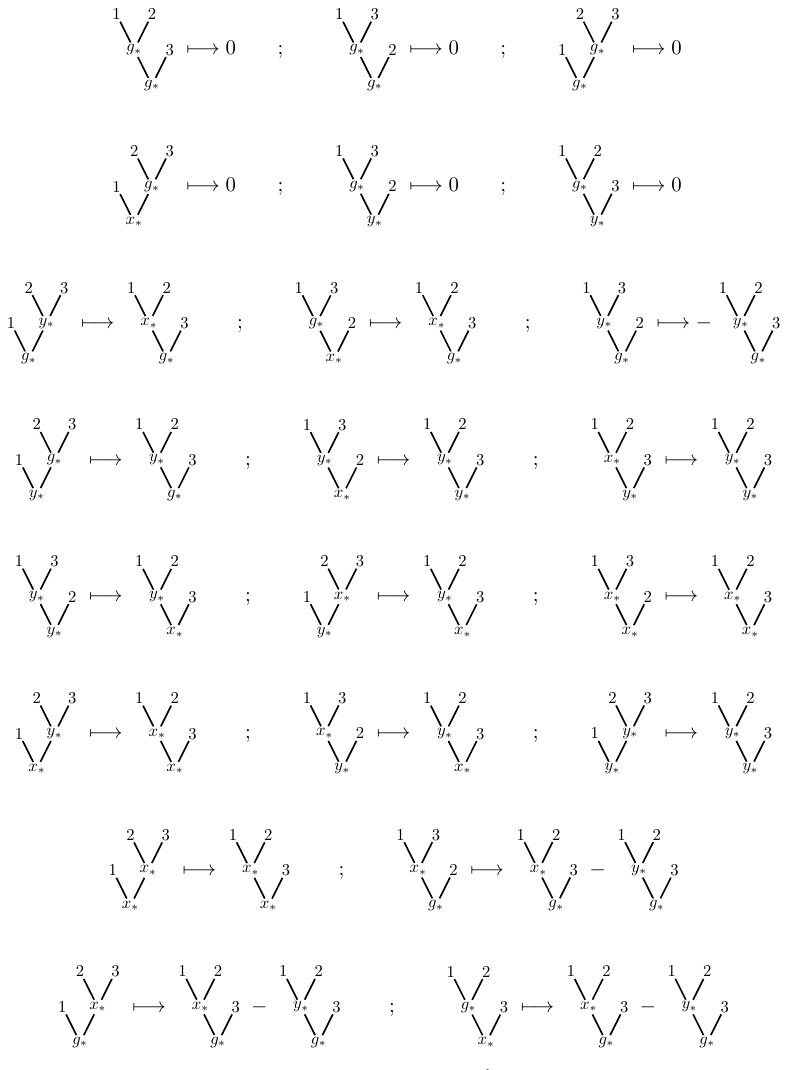}
  \end{center}
\end{figure}
\end{center}

\afterpage{\clearpage}
\begin{center}
  \begin{figure}[H]\caption{The ``rdlp'' rewriting system for $(\greg')^!$}\label{fg:StGB3}
    \begin{center}
    \includegraphics{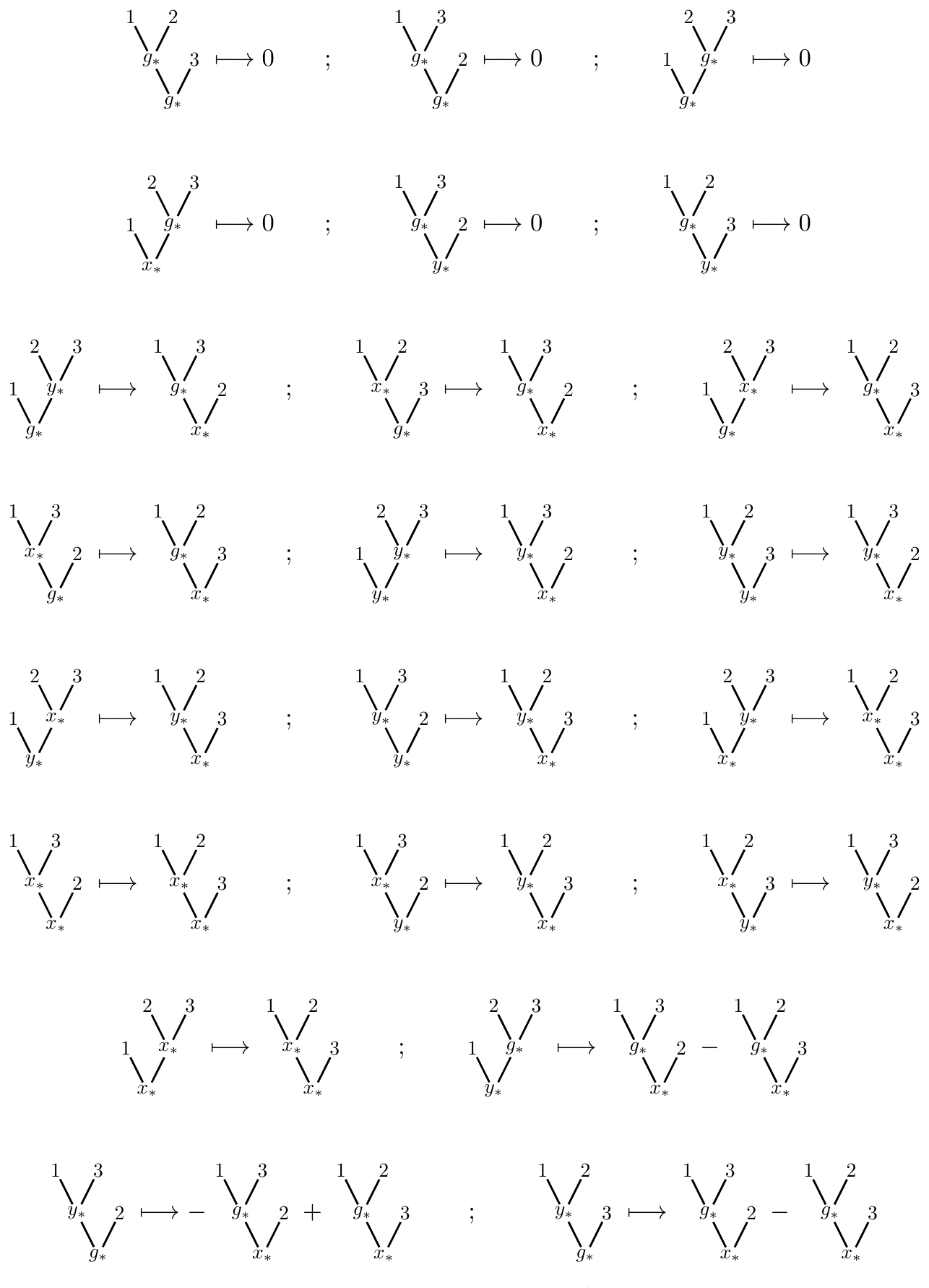}
    \end{center}
  \end{figure}
\end{center}

\afterpage{\clearpage}
\begin{center}
\begin{figure}[H]\caption{The computation of $(R\star^\Delta S)\star^\Delta T$}\label{fg:StComp1}
  \begin{center}
  \includegraphics{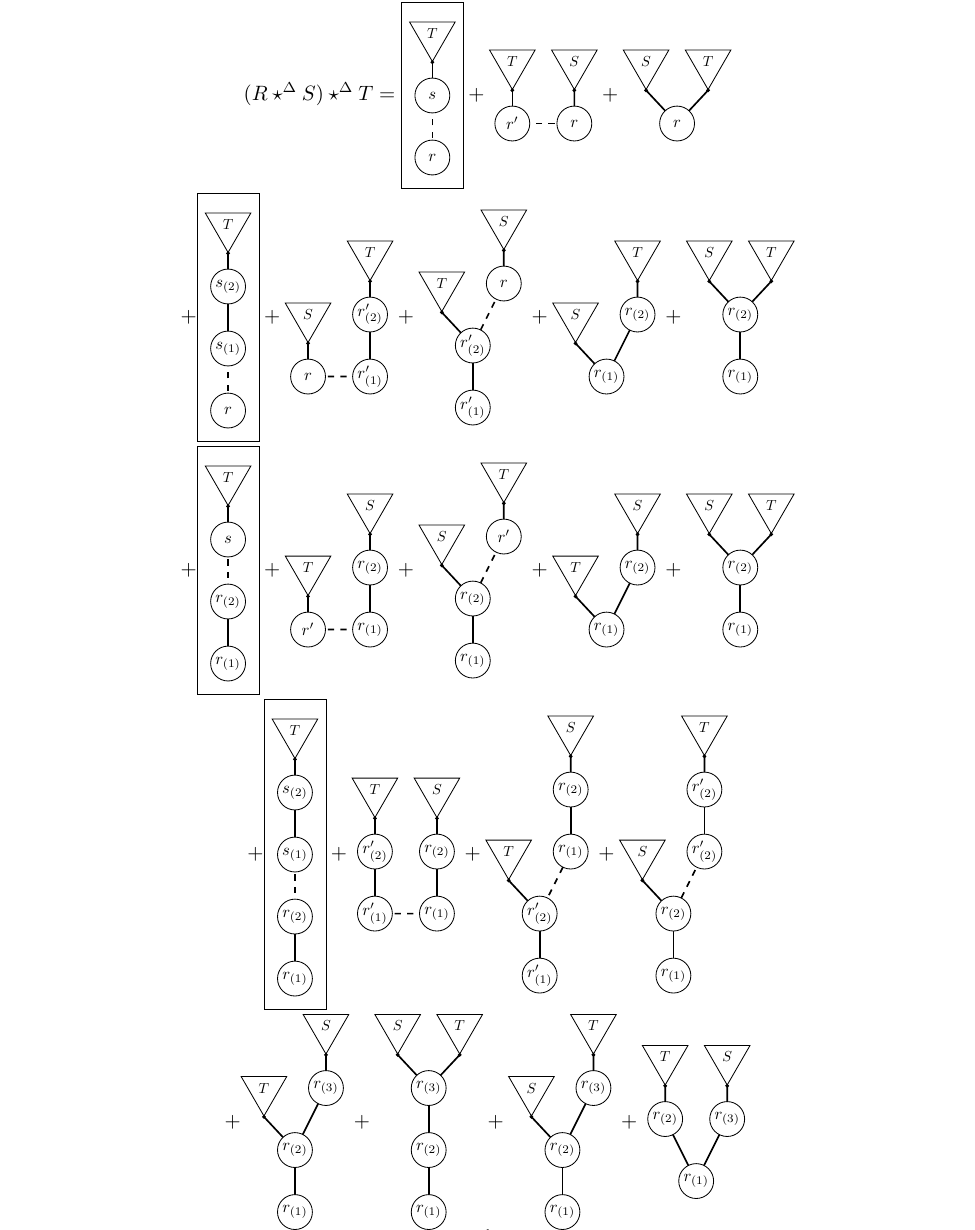}
  \end{center}
\end{figure}
\end{center}

\afterpage{\clearpage}
\begin{center}
  \begin{figure}[H]\caption{The ``dlp'' rewriting system for $(\G^C)^!$}\label{fg:StGBg1}
    \begin{center}
    \includegraphics{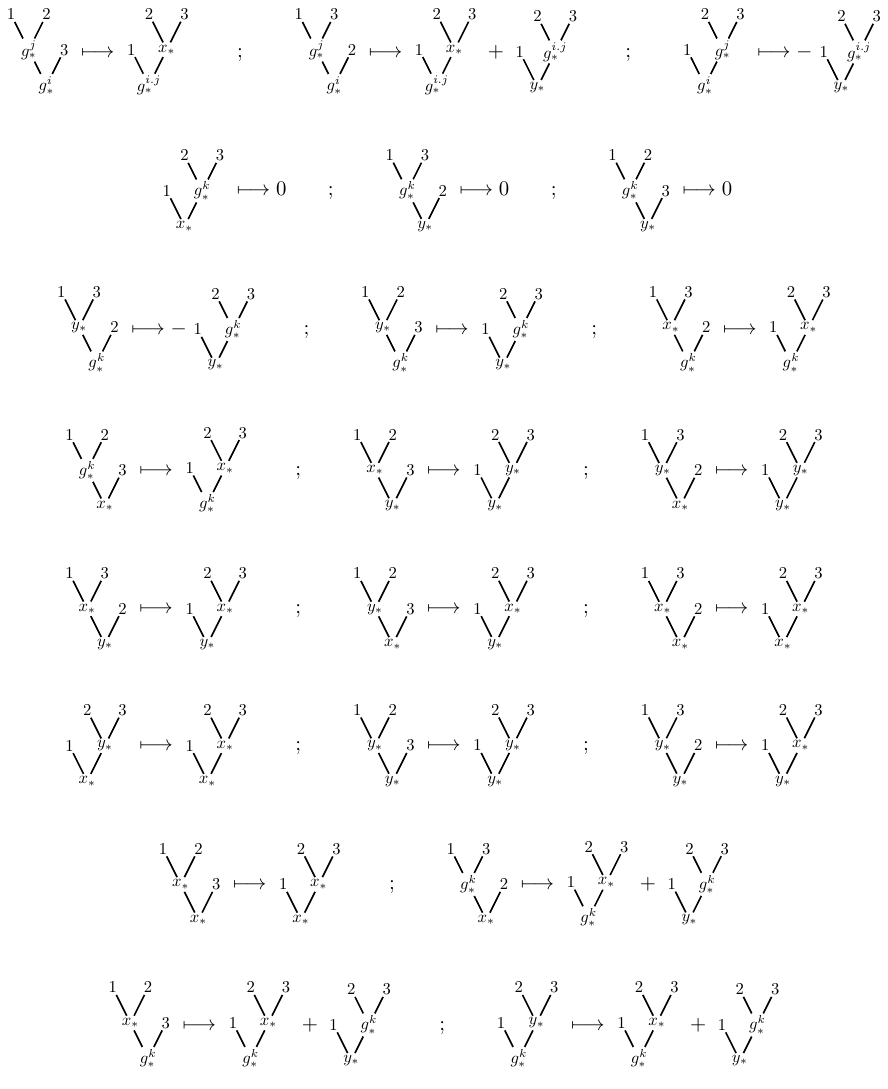}
    \end{center}
  \end{figure}
\end{center}

\afterpage{\clearpage}
\begin{center}
  \begin{figure}[H]\caption{The ``wprdl'' rewriting system for $(\G^C)^!$}\label{fg:StGBg2}
    \begin{center}
    \includegraphics{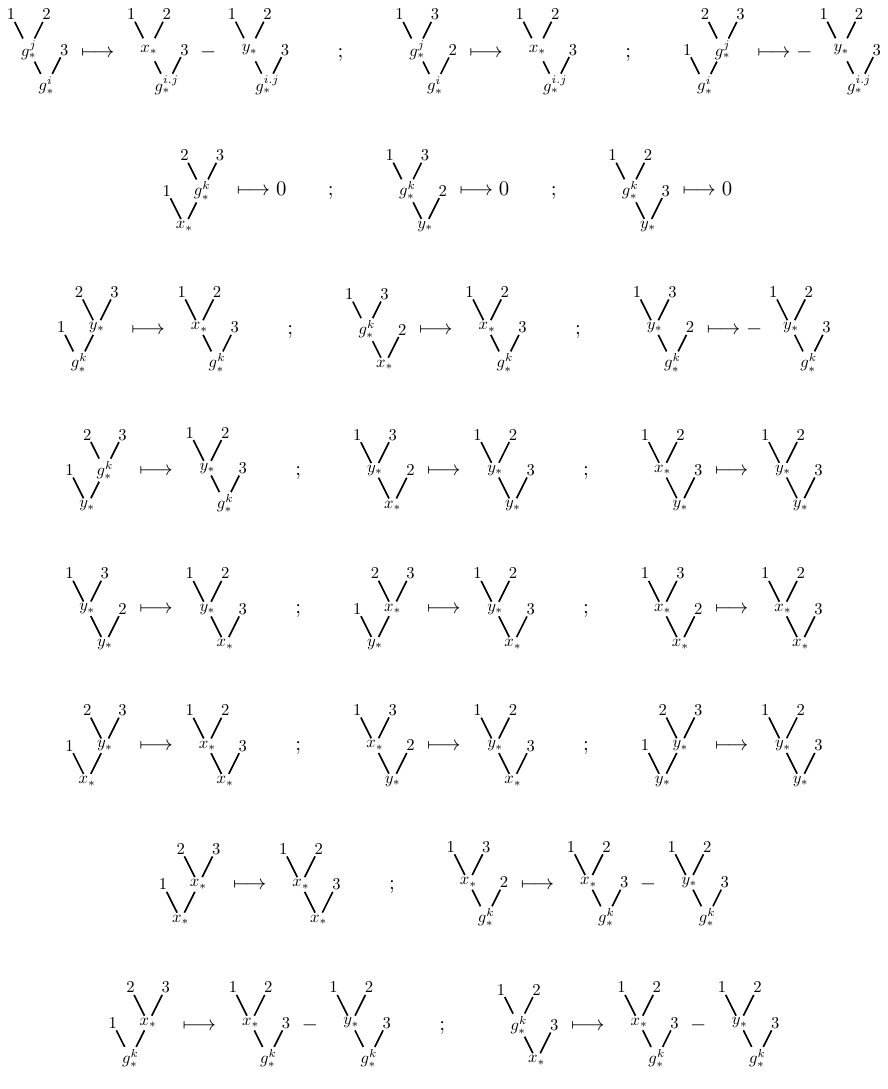}
    \end{center}
  \end{figure}
\end{center}

\afterpage{\clearpage}
\begin{center}
  \begin{figure}[H]\caption{The ``rdlp'' rewriting system for $(\G^C)^!$}\label{fg:StGBg3}
    \begin{center}
    \includegraphics{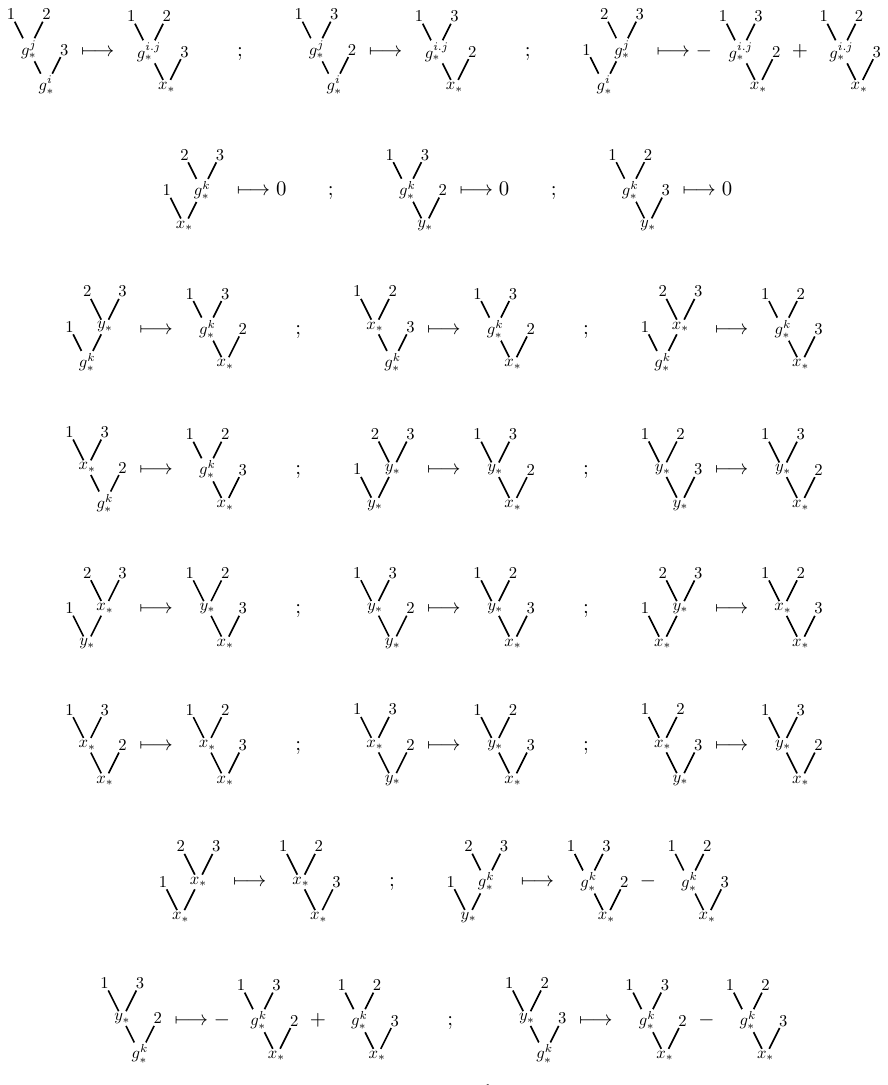}
    \end{center}
  \end{figure}
\end{center}


\newpage

\bibliographystyle{plain}
\bibliography{Bibly}

\begin{thebibliography}{10}

\bibitem{Species}
F.~Bergeron, G.~Labelle, and P.~Leroux.
\newblock {\em {Combinatorial Species and Tree-like Structures}}.
\newblock Cambridge University Press, 1998.

\bibitem{AlgComp}
Murray~R. Bremner and Vladimir Dotsenko.
\newblock {\em {Algebraic Operads: An Algorithmic Companion}}.
\newblock Chapman and Hall / CRC Press, 2016.

\bibitem{D2pl}
Tom Bridgeland.
\newblock {Geometry from Donaldson-Thomas invariants}.
\newblock \url{https://arxiv.org/abs/1912.06504v1}, 2019.

\bibitem{LSA}
Dietrich Burde.
\newblock {Left-symmetric algebras, or pre-Lie algebras in geometry and
  physics}.
\newblock {\em Open Mathematics}, 4(3):323–357, 2006.

\bibitem{FPL}
Frédéric Chapoton.
\newblock {Fine structures inside the PreLie operad}.
\newblock {\em Proceedings of the American Mathematical Society},
  141:3723--3737, 2013.

\bibitem{PreLie}
Frédéric Chapoton and Muriel Livernet.
\newblock {Pre-Lie algebras and the rooted trees operad}.
\newblock {\em International Mathematics Research Notices}, 8:395--408, 2001.

\bibitem{Free}
Vladimir Dotsenko.
\newblock {Freeness theorems for operads via Gröbner bases}.
\newblock {\em Séminaires et congrès (OPERADS 2009)}, 26:61--76, 2013.

\bibitem{FPRed}
Vladimir Dotsenko.
\newblock {Fine structures inside the PreLie operad revisited}, 2023.
\newblock \url{https://arxiv.org/abs/2306.08425}.

\bibitem{Hask}
Vladimir Dotsenko and Willem Heijltjes.
\newblock {Operadic Gröbner bases calculator}.
\newblock \url{https://irma.math.unistra.fr/~dotsenko/Operads.html}, 2014.

\bibitem{PBW}
Vladimir Dotsenko and Pedro Tamaroff.
\newblock {Endofunctors and Poincaré-Birkhoff-Witt theorems}.
\newblock {\em International Mathematics Research Notices}, 16:12670--12690,
  2021.

\bibitem{NSPrt}
Vladimir Dotsenko and Ualba\u \i~U. Umirbaev.
\newblock {An effective criterion for Nielsen-Schreier varieties}.
\newblock {\em International Mathematics Research Notices}, 2022.
\newblock to appear.

\bibitem{Greg}
C.~Flight.
\newblock {How many stemmata?}
\newblock {\em Manuscripta}, 34:122--128, 1990.

\bibitem{PLrightFree}
L.~Foissy.
\newblock {Finite-dimensional comodules over the {H}opf algebra of rooted
  trees}.
\newblock {\em Journal of Algebra}, 255(1):89--120, 2002.

\bibitem{Hgreg}
Walter~W. Greg.
\newblock {The Calculus of Variants: An Essay on Textual Criticism}.
\newblock {\em Oxford: Clarendon Press}, 1927.

\bibitem{GenGreg}
Matthieu Josuat-Vergès.
\newblock {Derivatives of the tree function}.
\newblock {\em The Ramanujan Journal}, 38:1--15, 2015.

\bibitem{sBr}
Tom Lada and Martin Markl.
\newblock {Symmetric brace algebras}.
\newblock {\em Applied Categorical Structures}, 4(13):351--370, 2005.

\bibitem{AlgOp}
Jean-Louis Loday and Bruno Vallette.
\newblock {\em {Algebraic Operads}}.
\newblock Springer Berlin, Heidelberg, 2012.

\bibitem{Hmaas}
Paul Maas.
\newblock {Textual Criticism}.
\newblock {\em Oxford: Clarendon Press}, 1927.
\newblock translated from the German by Barbara Flower.

\bibitem{Oeis}
OEIS.
\newblock {The On-Line Encyclopedia of Integer Sequences}.
\newblock Published electronically at \url{http://oeis.org}.

\bibitem{BRPL}
J.-M. Oudom and D.~Guin.
\newblock {On the {L}ie enveloping algebra of a pre-{L}ie algebra}.
\newblock {\em Journal of K-Theory. K-Theory and its Applications in Algebra,
  Geometry, Analysis \& Topology}, 2:147--167, 2008.

\bibitem{FreeLie}
A.~I. \v{S}ir\v{s}ov.
\newblock {Subalgebras of free {L}ie algebras}.
\newblock {\em Mat. Sbornik N.S.}, pages 441--452, 1953.

\bibitem{FreeLie2}
Ernst Witt.
\newblock {Die Unterringe der freien Lieschen Ringe}.
\newblock {\em Mathematische Zeitschrift}, 64:195--216, 1956.

\end{thebibliography}

\textsc{Institut de Recherche Mathématique Avancée, UMR 7501, Université de Strasbourg et CNRS, 
7 rue René-Descartes, 67000 Strasbourg CEDEX, France}\\
\emph{Email address:} laubie@unistra.fr

\end{document}